\documentclass[11pt]{article}
\usepackage{amsfonts}

\usepackage{graphics}
\usepackage{indentfirst}
\usepackage{cite}
\usepackage{latexsym}
\usepackage{amsmath,amsthm}
\usepackage{amssymb}
\usepackage[dvips]{epsfig}
\usepackage{amscd}
\usepackage{mathrsfs}

\usepackage{color}
\newtheorem{theorem}{Theorem}[section]
\newtheorem{remark}{Remark}[section]

\newtheorem{definition}{Definition}[section]
\newtheorem{lemma}[theorem]{Lemma}

\newtheorem{proposition}[theorem]{Proposition}

\newcommand{\n}{\rho}

\newcommand{\lm}{\lambda}

\renewcommand{\div}{ {\rm div }  }

\newcommand{\na}{\nabla }

\newcommand{\pa}{\partial}

\newcommand{\bt}{\begin{theorem}}
\newcommand{\bl}{\begin{lemma}}
\newcommand{\el}{\end{lemma}}
\newcommand{\et}{\end{theorem}}
\newcommand{\ga}{\gamma}

\newcommand{\OM}{\Omega}

\newcommand{\de}{\delta}
\newcommand{\ve}{\varepsilon}
\newcommand{\la}{\label}
\newcommand{\si}{\sigma}

\newcommand{\bn}{\begin{eqnarray}}
\newcommand{\en}{\end{eqnarray}}
\newcommand{\bnn}{\begin{eqnarray*}}
\newcommand{\enn}{\end{eqnarray*}}

\newcommand{\bnnn}{\begin{eqnarray*}}
\newcommand{\ennn}{\end{eqnarray*}}
\newcommand{\ben}{\begin{enumerate}}
\newcommand{\een}{\end{enumerate}}

\newcommand{\du}{\dot{u}}

\newcommand{\ba}{\begin{aligned}}
\newcommand{\ea}{\end{aligned}}
\newcommand{\be}{\begin{equation}}
\newcommand{\ee}{\end{equation}}

\def\p{\partial}
\def\norm[#1]#2{\|#2\|_{#1}}

\def\lap{\triangle}

\def\lam{\lambda}

\def\o{\omega}
\def\rr{\mathbb{R}^2}

\makeatletter      
\@addtoreset{equation}{section}
\makeatother       

\title{Global Existence and Incompressible Limit of the Cauchy Problem for 2D Compressible Navier-Stokes Equations with Large Bulk Viscosity and Large Initial Data}

\date{}

\author{$\text{Qinghao L{\small EI}}^{a,b}, \text{Chengfeng X{\small{IONG}}}^{a,b}\thanks{Email addresses:  leiqinghao22@mails.ucas.ac.cn (Q. H. Lei), xiongchengfeng20@mails.ucas.ac.cn (C. F. Xiong) }$\\
a. School of Mathematical Sciences,\\ University of Chinese Academy of Sciences,
Beijing 100049, P. R. China;\\
b. Institute of Applied Mathematics,\\ Academy of Mathematics and Systems Science, \\
Chinese Academy of Sciences, Beijing 100190, P. R. China}

\begin{document}
\maketitle

\begin{abstract}
This paper investigates the Cauchy problem for the barotropic compressible Navier-Stokes equations in $\mathbb{R}^2$ with the constant state as far field, which may be vacuum or non-vacuum.
Under the assumption of a sufficiently large bulk viscosity coefficient,
we establish the global existence and large time behavior of weak, strong, and classical solutions.
It should be mentioned that this result is obtained without any restrictions on the size of the initial data.
Moreover, we demonstrate that as the bulk viscosity coefficient tends to infinity, the solutions of the compressible Navier-Stokes equations converge to those of the inhomogeneous incompressible Navier-Stokes equations.
The incompressible limit of the weak solutions holds even without requiring the initial velocity to be divergence-free. 
\end{abstract}

\section{Introduction and main results}
We study the two-dimensional barotropic compressible
Navier-Stokes equations which read as follows:
\be\ba\la{ns}
\begin{cases}
  \rho_t + \div(\rho u) = 0,\\
  (\n u)_t + \div(\n u\otimes u) -\mu \Delta u - (\mu + \lm)\na\div u
    +\na P = 0,
\end{cases}
\ea\ee
where $t \ge 0$ is time, $x \in \OM \subset \rr$ is the spatial coordinate,
$\n=\n(x,t)$ and $u(x,t)=(u^1(x,t),u^2(x,t))$ represent the 
density and velocity of the compressible flow respectively, and the 
pressure $P$ is given by
\be\la{i1}
P=R\n^\ga,
\ee
with constants $R>0,\ga>1$. Without loss of generality, it is assumed that
$R=1$. The shear viscosity coefficient $\mu$ and bulk viscosity coefficient $\lam$ satisfy
the physical restrictions:
\be\la{i2}
\mu>0,\quad \mu+\lam\geq 0.
\ee
For later purpose, we set 
\be\la{nu}
\nu := 2\mu + \lam,
\ee
which together with (\ref{i2}) yields that 
\be\la{numu}\nu\geq\mu.\ee
Let $\OM = \rr$ and $\tilde{\rho}$ be a fixed nonnegative constant.
We look for the solutions $ \left(\n(x,t),u(x,t)\right) $ to the Cauchy problem for (\ref{ns}) with the far-field behavior:
\be\la{}\ba
u(x,t)\rightarrow 0,\quad \rho(x,t)\rightarrow \tilde{\rho} \ge 0,  \quad\mbox{
	as }\, |x|\rightarrow\infty,
\ea\ee
and the initial data
\be\la{i3}
\n(x,0)=\n_0(x),\quad \n u(x,0)=\n_0u_0(x), \quad x \in \rr.
\ee

Moreover, when $\tilde{\n}=0$, it is obvious that the total mass of smooth enough solutions
of (\ref{ns}) is conserved through the evolution, that is, for all
$t>0$,
\be\ba\la{conv}
\int_{\rr} \n dx = \int_{\rr}\n_0 dx.
\ea\ee
Without loss of generality, when $\tilde{\n}=0$, we shall assume that 
\be\la{rho0}\ba
\int \n_0 dx =1,
\ea\ee
which implies that there exists a positive constant $N_0$ such that
\be\la{rho00}\ba
\int_{B_{N_0}} \n_0 dx \ge \frac{1}{2}\int \n_0 dx =\frac{1}{2}.
\ea\ee

There exists a huge literature concerning the theory of weak and classical solutions to the compressible Navier-Stokes equations.
The one-dimensional problem has been thoroughly investigated (see \cite{H4,KS,S1,S2} and references therein).
For the multi-dimensional problem, the local existence and uniqueness of classical 
solutions in the absence of vacuum were first proved by Nash \cite{N} and Serrin \cite{S}.
Recently, for the case in which the initial density contains vacuum and may vanish in open sets,
the local existence and uniqueness of strong and classical solutions have been rigorously established in \cite{CCK,CK,CK2,SS,LLL}
and references therein.
In particular, for the two-dimensional Cauchy problem, Li and Liang \cite{LLL} proved
the existence and uniqueness of local strong and classical solutions to (\ref{ns})
with vacuum states at infinity.
As for global solutions, 
Matsumura and Nishida \cite{MN1} first established the existence 
of global classical solutions for initial data sufficiently 
close to a non-vacuum equilibrium in a Sobolev space $H^s$.
Subsequently, Hoff \cite{H1,H2,H3} studied the problem with discontinuous
initial data with a new type of a priori estimates of the material derivative.
A major breakthrough was achieved 
by Lions \cite{L2}, who proved the global existence of weak solutions 
for arbitrarily large initial data under the assumptions 
that the initial energy is finite and $\ga$ is sufficiently large.
Later, the range of the adiabatic exponent $\ga$ was relaxed by Feireisl-Novotn\'y-Petzeltov\'a \cite{FNP}.
Recently, for the case that the initial density allows vacuum,
Huang-Li-Xin\cite{HLX2} proved the global
existence and uniqueness of classical solutions to the Cauchy problem
in three spatial dimensions provided that the initial energy is sufficiently small
while allowing for large oscillations.
Subsequently, Li-Xin\cite{LX2} extended these results to the two-dimensional Cauchy problem under the same small energy assumptions.

More recently, Danchin-Mucha \cite{DM} established the global existence of weak solutions under sufficiently large bulk viscosity and a restriction of the upper bound of $\nu^{1/2} \|\div u_0\|_{L^2}$.
Moreover, they proved that when the bulk viscosity tends to infinity,
these weak solutions converge to those of the inhomogeneous incompressible Navier-Stokes system.
Subsequently, Liao-Zodji \cite{LZ} extended the result of \cite{DM} to the two-dimensional whole space with non-vacuum far-field density.
In our recent work \cite{LX3}, we generalized the result of \cite{DM},
requiring only that the bulk viscosity coefficient be sufficiently large to establish the global existence of solutions to (\ref{ns}).
Furthermore, we do not even require the initial velocity field to be divergence-free
when deriving the singular limit from the compressible Navier-Stokes equations to the inhomogeneous incompressible Navier-Stokes equations.
In this paper, we address the Cauchy problem, establishing the global existence and large time behavior of weak, strong, and classical solutions
without imposing any restriction on $\nu^{1/2} \|\div u_0\|_{L^2}$. 
Furthermore, we prove that in the limit as the bulk viscosity tends to infinity,
these solutions converge to solutions of the
inhomogeneous incompressible Navier-Stokes equations.

Before stating the main results, we first explain the notations
and conventions used throughout this paper. We denote
\be\ba\la{i4}
  \int f dx = \int_{\rr} fdx.
\ea\ee
For $1\leq r\leq\infty$, we also denote the standard Lebesgue and Sobolev
spaces as follows:
\be\la{i5}\ba
{\left\{\begin{array}{ll}
L^r=L^r(\rr ),\quad D^{k,r}=D^{k,r}(\rr )=\{v\in L^1_\mathrm{loc}(\rr )| \nabla ^k v\in L^r(\rr )\}, \\ D^1=D^{1,2},\quad W^{k,r} = W^{k,r}(\rr ) , \quad H^k = W^{k,2}.
\end{array}\right.}
\ea\ee
The initial total energy is defined as follows:
\be\la{e0}\ba
E_0 := \int \frac{1}{2} \rho_0 |u_0|^2 + H(\n_0) dx,
\ea\ee
where $H(\n)$ denotes the potential energy density given by
\bnn
H(\n) \triangleq \n \int_{\tilde{\rho}}^{\n} \frac{P(s)-P(\tilde{\rho})}{s^2} ds.
\enn
It is clear that
\be\la{qkjsn}\ba
\begin{cases}
H(\n) = \frac{1}{\ga-1} \n^\ga , & \quad \text{ if } \tilde{\n}=0, \\
\frac{1}{c(\hat{\n},\tilde{\n})} (\n - \tilde{\n})^2 \le H(\n)
\le c(\hat{\n},\tilde{\n}) (\n - \tilde{\n})^2, & \quad \text{ if } \tilde{\n} > 0, \  0 \le \n \le \hat{\n},
\end{cases}
\ea\ee
for some positive constant $c(\hat{\n},\tilde{\n})$.

Then we provide the definition of weak and strong solutions to (\ref{ns}).
\begin{definition}
If $(\n,u)$ satisfies (\ref{ns}) in the sense of distribution, then we call $(\n,u)$ a weak solution.
Moreover, for a weak solution if
all derivatives involved in (\ref{ns}) are regular distributions
and equations (\ref{ns}) hold almost everywhere in 
$\rr\times(0,T)$, then $(\n,u)$ is called a strong solution.
\end{definition}

The first main result concerning the existence of weak solutions can be described as follows:
\begin{theorem}\la{th0}
Assume the initial data $(\n_0,u_0)$ satisfy
\be \la{wsol1}\ba
\n_0 \ge 0, \quad \n_0 |u_0|^2 \in L^1,  \quad u_0 \in D^1,
\ea\ee
and for some $a>1$,
\be\la{wsol01}\ba
\begin{cases}
{\bar{x}}^a \rho_0 \in L^1, \ \rho_0 \in L^\infty,  \quad &\mathrm{ if \  }  \tilde{\n}=0,\\
\n_0-\tilde{\rho} \in L^2 \cap L^\infty,  \quad &\mathrm{ if \  }  \tilde{\n}>0,
\end{cases}
\ea\ee
where
\be\la{wsol02}\ba
{\bar{x}}\triangleq (e+|x|^2)^{\frac{1}{2}} \log ^2 (e+|x|^2).
\ea\ee

\noindent\textbf{\textup{(1) Vacuum far-field density ($\tilde{\n}=0$):}}
There exists a positive constant $\nu_1$ depending only on
$N_0$, $\ga$, $\mu$, $a$, $E_0$, $\|{\bar{x}}^a \n_0\|_{L^1}$, $\|\n_0\|_{L^\infty}$, and $\| \na u_0\|_{L^2}$,
such that when $\nu \ge \nu_1$,
the problem \eqref{ns}--\eqref{i3} with $\tilde{\n}=0$ has at least one
global weak solution $(\n,u)$ in $\rr \times (0,\infty)$ satisfying
\be\la{wsol2}\ba
0\le \n(x,t) \leq 2 \left( 1 + \| \n_0 \|_{L^\infty} \right),
\quad \mathrm{for\ any\ }(x,t)\in \rr \times[0,\infty),
\ea\ee
and for any $1 \le p <\infty$,
\be\la{wsol3}\ba
\begin{cases}
\rho\in L^\infty(\rr \times (0,T)) \cap C([0,T];L^p), \quad {\bar{x}}^a \rho \in L^\infty(0,T;L^1) \\ 
\na u\in L^\infty(0,T;L^2), t^{1/2}u_t \in L^2(0,T;L^2), t^{1/2} \na u \in L^\infty(0,T;L^p).
\end{cases}
\ea\ee
Moreover, there exists a positive constant $N_1$ depending on
$N_0$, $\| {\bar{x}}^a \rho_0 \|_{L^1}$, and $E_0$, such that
\be\la{wsol30}\ba
\inf_{0 \le t \le T} \int_{ B_{N_1 (1+t)} }\rho (x,t) dx \ge \frac{1}{4}.
\ea\ee
Furthermore, $(\n,u)$ has the following decay rates, that is, for $t \ge 1$,
\be\la{wsol300}\ba
\begin{cases}
\| \nabla u(\cdot,t) \|_{L^{p}} \leq C(p) t^{-1+\frac{1}{p}}, \ & \textnormal{ for } p \in [2,\infty), \\
\| P(\cdot,t) \|_{L^{r}} \leq C(r) t^{-1+\frac{1}{r}}, \ & \textnormal{ for } r \in (1,\infty), \\
\| \sqrt{\n} \dot{u}(\cdot,t) \|_{L^2} \le C t^{-1},
\end{cases}
\ea\ee
where $C(z)$ depends on $z$ besides $N_0$, $\ga$, $\mu$, $E_0$, $\|{\bar{x}}^a \n_0\|_{L^1}$, $\|\n_0\|_{L^\infty}$, and $\| \na u_0\|_{L^2}$.

\noindent\textbf{\textup{(2) Non-vacuum far-field density ($\tilde{\n}>0$):}}
There exists a positive constant $\nu_2$ depending only on
$\ga$, $\mu$, $E_0$, $\tilde{\n}$, $\|\n_0\|_{L^\infty}$, and $\| \na u_0\|_{L^2}$,
such that when $\nu \ge \nu_2$,
the problem \eqref{ns}--\eqref{i3} with $\tilde{\n}>0$ has at least one
global weak solution $(\n,u)$ in $\rr \times (0,\infty)$ satisfying
\be\la{2wsol2}\ba
0\le \n(x,t) \leq 2 \left( 1 + \| \n_0 \|_{L^\infty} \right) e^{\frac{1}{\ga} \tilde{\n}^\ga},
\quad \mathrm{for\ any\ }(x,t)\in \rr \times[0,\infty).
\ea\ee
and for any $0<T<\infty$, $1 \le p < \infty $, and $2 \le s <\infty$,
\be\la{wsol4}\ba
\begin{cases}
\rho-\tilde{\n} \in L^\infty(\rr \times (0,T)) \cap C([0,T];L^s), \\ 
u\in L^\infty(0,T;H^1), t^{1/2}u_t \in L^2(0,T;L^2), t^{1/2} \na u \in L^\infty(0,T;L^p).
\end{cases}
\ea\ee
Furthermore, the following large-time behavior holds:
\be\la{cpwsol4}\ba
\lim_{t \to \infty} \left( \| \n(\cdot,t) - \tilde{\n} \|_{L^s} + \| \na u(\cdot,t) \|_{L^r} \right) = 0,
\ea\ee
for any $s \in (2,\infty)$ and $r \in [2,\infty)$.
\end{theorem}

\begin{theorem}\la{th01}
Let $\tilde{\n}=0$, and suppose the initial data $(\n_0,u_0)$ satisfy the assumptions of Theorem \ref{th0} 
in this case.
For $\nu_1$ determined in Theorem \ref{th0}, when $\nu \ge \nu_1$,
we denote by $(\n^{\nu},u^{\nu})$ the weak solution of 
\eqref{ns}--\eqref{i3} established in Theorem \ref{th0}.
Then, as $\nu$ tends to $\infty $, there exists a subsequence of $(\n^{\nu},u^{\nu})$ 
that converges to the solution $(\n, u)$ to the following
inhomogeneous incompressible Navier-Stokes equations:
\be\la{isol2}\ba
\begin{cases}
\n_t+\div(\n u)=0,\\
(\n u)_t+\div(\n u\otimes u) -\mu \Delta u + \na \pi =0, \\
\div u=0,
\end{cases} 
\ea\ee
with initial data $\n(\cdot,0)=\n_0,\ \n u(\cdot,0)=m_0:=\n_0 u_0$.
Furthermore, $(\n,u)$ satisfies 
for any $0<T<\infty$, $0<R<\infty$, $2<r<\infty$ and $1\le p <\infty$,
\be\la{isol3}\ba
\begin{cases}
\rho\in L^{\infty}(\rr \times (0,T)) \cap C([0,T];L^p), \quad {\bar{x}}^a \rho \in L^\infty(0,T;L^1), \\ 
u\in L^2(0,T;L^2(B_R)), \\
\na u,\ \sqrt{t} \sqrt{\n} \du,\ \sqrt{t} \na \pi,\ \sqrt{t} \na^2 u,\ t \na \du \in L^2(\rr \times (0,T)), \\
\sqrt{t} \na u,\ t \sqrt{\n} \du,\ t \na \pi,\ t \na^2 u \in L^\infty(0,T;L^2), \\
\sqrt{t} \pi \in L^2(0,T;L^r), \quad t \pi \in L^\infty(0,T;L^r).
\end{cases} 
\ea\ee
Moreover, if the initial data $(\n_0,u_0)$ further satisfy for some $a>1$,
\be\la{ws}\ba
{\bar{x}}^a \rho_0 \in L^\infty(\rr),\ \div u_0=0,
\ea\ee
and the compatibility condition
\be\la{lws00}\ba
- \mu \Delta u_0 + \na \pi_0 = \sqrt{\n_0} g_1, \ \textnormal{ for } u_0 \in L^2(\rr),\ \na u_0 \in L^1(\rr),\ g_1 \in L^2(\rr),
\ea\ee
then the entire sequence $(\n^{\nu},u^{\nu})$ converges to the unique global solution of 
\eqref{isol2},
and $(\n,u)$ satisfies for any $0<T<\infty$, $2<r<\infty$, $1\le p <\infty$, and $0<R<\infty$,
\be\la{lws1}\ba
\begin{cases}
\rho \in C([0,T];L^p), \quad
{\bar{x}}^a \n \in L^\infty(0,T;L^1) \cap L^{\infty}(\rr \times (0,T)), \\
u\in L^\infty(B_R \times (0,T)), \quad \na u, \ \sqrt{\n} \du,\ \na^2 u,\ \na \pi \in L^\infty(0,T;L^2), \\ 
\na \du \in L^2(\rr \times (0,T)), \quad \pi \in L^\infty(0,T;L^r).
\end{cases} 
\ea\ee
\end{theorem}

\begin{theorem}\la{th001}
Let $\tilde{\n}>0$, and assume the initial data $(\n_0,u_0)$ satisfy the hypotheses of Theorem \ref{th0} in this case.
For $\nu_2$ specified in Theorem \ref{th0}, when $\nu \ge \nu_2$,
we denote $(\n^{\nu},u^{\nu})$ as the weak solution to \eqref{ns}--
\eqref{i3} given by Theorem \ref{th0}.
Then, as $\nu$ tends to $\infty$, the solution sequence $(\n^{\nu},u^{\nu})$ admits a subsequence that
converges to the solution $(\n, u)$ to \eqref{isol2},
which satisfies for any $0<T<\infty$ and $2\le s <\infty$,
\be\la{lisol3}\ba
\begin{cases}
\rho -\tilde{\n} \in L^{\infty}(\rr \times (0,T)) \cap C([0,T];L^s), \\ 
u\in L^2(0,T;H^1), \\
\sqrt{t} \na^2 u,\ \sqrt{t} \na \pi,\ t \sqrt{\n} u_t,\ t^2 \na u_t \in L^2(\rr \times (0,T)), \\
\sqrt{t} \na u,t \na \pi,\ t \na^2 u,\ t^2 \sqrt{\n} u_t \in L^\infty(0,T;L^2).
\end{cases} 
\ea\ee
If the initial data $(\n_0,u_0)$ additionally satisfy
\be\la{0lws}\ba
\div u_0=0,
\ea\ee
then the entire sequence $(\n^{\nu},u^{\nu})$ converges to the unique global solution 
$(\n,u)$ to \eqref{isol2}, which satisfies for any $0<T<\infty$, $1\le p <\infty$ and $2\le s <\infty$,
\be\la{0lws1}\ba
\begin{cases}
\rho -\tilde{\n} \in L^{\infty}(\rr \times (0,T)) \cap C([0,T];L^s), \\
u\in L^\infty(0,T;H^1), \quad \sqrt{\n} u \in C([0,T];L^2) \\
\sqrt{\n} u_t,\ \na^2 u,\ \na \pi,\ \sqrt{t} \na u_t \in L^2(\rr \times (0,T)), \\ 
\sqrt{\n} u_t,\ \sqrt{t} \na \pi,\ \sqrt{t} \na^2 u \in L^\infty(0,T;L^2) \cap L^2(0,T;L^p).
\end{cases} 
\ea\ee
\end{theorem}

\begin{remark}\la{ctoi1}
It is noted that we prove the convergence of the compressible Navier-Stokes equations to the incompressible Navier-Stokes equations
without imposing the restrictive condition ${\rm div} u_0 = 0$.
In fact, as shown in \cite[Theorem 2.1]{L1}, the inhomogeneous incompressible Navier-Stokes equations
 \eqref{isol2} admit global solutions even when $\div u_0 \neq 0$.
\end{remark}

\begin{remark}\la{ctoi2}
For the solution $(\n,u)$ of (\ref{isol2}) satisfying (\ref{isol3}) or (\ref{lisol3})
with the initial data $(\n_0,m_0) = (\n_0,\n_0u_0)$.
This meaning that
for all $T\in (0,\infty)$, 
$(-\Delta)^{-1/2}\na^\bot\cdot \n u \in C([0,T];L^2_w)$ with 
$(-\Delta)^{-1/2}\na^\bot\cdot \n u(\cdot,0) = (-\Delta)^{-1/2}\na^\bot\cdot m_0$.
The reason why we cannot  obtain the time-continuitiy of $\n u $ is that 
$ \na \pi $ is only in $L^2(\tau,T;L^2)$ rather than $L^2(0,T;L^2)$.
For a detailed analysis of the time-continuity of properties of $\n u$ see \cite[Theorem 2.2]{L1}.
\end{remark}

\begin{theorem}\la{th1}
In addition to the assumption of the initial data $(\n_0,u_0)$ in Theorem \ref{th0}, we assume further that for some $q>2$,
\be\la{ssol1}\ba
\begin{cases}
{\bar{x}}^a \rho_0 \in H^1 \cap W^{1,q},     \quad &\mathrm{ if \  }  \tilde{\n}=0,\\
\n_0-\tilde{\rho} \in H^1 \cap W^{1,q},      \quad &\mathrm{ if \  }  \tilde{\n}>0.
\end{cases}
\ea\ee

\noindent\textbf{\textup{(1) Vacuum far-field density ($\tilde{\n}=0$):}}
For $\nu_1$ determined in Theorem $\ref{th0}$, when $\nu \ge \nu_1$,
the problem \eqref{ns})--\eqref{i3} with $\tilde{\n}=0$ has a unique global strong solution $(\n,u)$ in $\rr \times (0,\infty)$
satisfying \eqref{wsol2}, \eqref{wsol30}, \eqref{wsol300}, and
\be\la{ssol4}\ba
\begin{cases}	
\rho \in C([0,T];L^1 \cap H^1\cap W^{1,q} ),\\ 
{\bar{x}}^a\rho \in L^\infty ( 0,T ;L^1\cap H^1\cap W^{1,q} ),\\ 
\sqrt{\rho } u,\,\nabla u,\, {\bar{x}}^{-1}u, \, \sqrt{t} \sqrt{\rho } u_t \in L^\infty (0,T; L^2 ) , \\ 
\nabla u\in L^2(0,T;H^1)\cap L^{(q+1)/q}(0,T; W^{1,q}), \\ 
\sqrt{t}\nabla u\in L^2(0,T; W^{1,q} ) , \\ 
\sqrt{\rho } u_t, \, \sqrt{t}\nabla u_t ,\, \sqrt{t} {\bar{x}}^{-1}u_t\in L^2({\mathbb {R}^2 }\times (0,T)).
\end{cases} 
\ea\ee

\noindent\textbf{\textup{(2) Non-vacuum far-field density ($\tilde{\n}>0$):}}
For $\nu_2$ as in Theorem $\ref{th0}$, when $\nu \ge \nu_2$,
the problem \eqref{ns})--\eqref{i3} with $\tilde{\n}>0$ has a unique global strong solution $(\n,u)$ in $\rr \times (0,\infty)$
satisfying \eqref{2wsol2}, \eqref{cpwsol4}, and
\be\la{ssol5}\ba
\begin{cases}	
\rho-\tilde{\n}\in C([0,T];W^{1,q} ), \quad \n_t\in L^\infty(0,T;L^2), \\ 
u\in L^\infty(0,T; H^1) \cap L^{(q+1)/q}(0,T; W^{2,q}), \\ 
t^{1/2}u \in L^2(0,T; W^{2,q}) \cap L^\infty(0,T;H^2), \\
t^{1/2}u_t \in L^2(0,T;H^1), \\
\n u\in C([0,T];L^2), \quad \sqrt{\n} u_t\in L^2(\rr \times(0,T)),
\end{cases} 
\ea\ee
for any $0<T<\infty$.
\end{theorem}

\begin{theorem}\la{th2}
Suppose that the initial data $(\n_0,u_0)$ satisfy \eqref{wsol1}, 
\eqref{wsol01}, \eqref{ssol1}
and for some $q>2$ and $\delta_0 \in (0,1)$,
\be\la{csol1}\ba
\begin{cases}
\na^2\n_0,\ \na^2P(\n_0)\in L^2 \cap L^q,\quad {\bar{x}}^{\delta_0}\na^2\n_0,
\ {\bar{x}}^{\delta_0}\na^2P(\n_0),\ \na^2u_0 \in L^2, &\mathrm{ if \  }  \tilde{\n}=0, \\
\na^2\n_0,\ \na^2P(\n_0)\in L^2 \cap L^q, \quad \na^2u_0 \in L^2,  &\mathrm{ if \  }  \tilde{\n}>0.
\end{cases}
\ea\ee
Assume further that the following compatibility condition holds:
\be\la{csol2}\ba
-\mu\lap u_0 - (\mu + \lm )\nabla \div u_0 +  \nabla P(\n_0)=\sqrt{\n_0}g,
\ea\ee
for some $g \in L^2$.

\noindent\textbf{\textup{(1) Vacuum far-field density ($\tilde{\n}=0$):}}
For $\nu_1$ determined in Theorem $\ref{th0}$, when $\nu \ge \nu_1$,
the problem \eqref{ns})--\eqref{i3} with $\tilde{\n}=0$ has a unique global classical solution $(\n,u)$ in $\rr \times (0,\infty)$
satisfying \eqref{wsol2}, \eqref{wsol30}, \eqref{wsol300}, \eqref{ssol4}, and
\be\la{csol3}\ba
\begin{cases}
\nabla ^2\rho , \,\,\nabla ^2 P(\rho )\in C([0,T];L^2\cap L^q ), \\ 
{\bar{x}}^{\delta _0}\nabla ^2 \rho ,\,\, {\bar{x}}^{\delta _0} \nabla ^2 P( \rho ), \, \nabla ^2 u \in L^\infty ( 0,T ;L^2 ) ,\\ 
\sqrt{\rho } u_t, \,\sqrt{t} \nabla u_t,\,\sqrt{t} {\bar{x}}^{-1} u_t,\, t\sqrt{\rho }u_{tt}, \,t \nabla ^2 u_t\in L^\infty (0,T; L^2),\\ 
t\nabla ^3 u\in L^\infty (0,T; L^2\cap L^q), \,\\ 
\nabla u_t,\, {\bar{x}}^{-1}u_t,\, t\nabla u_{tt},\, t{\bar{x}}^{-1}u_{tt}\in L^2(0,T;L^2), \\ 
t \nabla ^2(\rho u)\in L^\infty (0,T;L^{(q+2)/2}).
\end{cases} 
\ea\ee

\noindent\textbf{\textup{(2) Non-vacuum far-field density ($\tilde{\n}>0$):}}
For $\nu_2$ as in Theorem $\ref{th0}$, when $\nu \ge \nu_2$,
the problem \eqref{ns})--\eqref{i3} with $\tilde{\n}>0$ has a unique global classical solution $(\n,u)$ in $\rr \times (0,\infty)$
satisfying \eqref{2wsol2}, \eqref{cpwsol4}, \eqref{ssol5}, and
\be\la{csol4}\ba
\begin{cases}
(\rho-\tilde{\n},P(\n) -P(\tilde{\n}) )\in C([0,T];W^{2,q} ),\quad  (\n_t,P_t)\in L^\infty(0,T;H^1), \\ 
(\n_{tt},P_{tt})\in L^2(0,T;L^2), \\
u\in L^\infty(0,T; H^2) \cap L^2(0,T;H^3), \\ 
\na u_t, \na^3u \in L^{(q+1)/q}(0,T;L^q), \\
t^{1/2}\na^3u \in L^\infty(0,T;L^2)\cap L^2(0,T; L^q), \\
t^{1/2}u_t\in L^\infty(0,T;H^1) \cap L^2(0,T;H^2), \\ 
t^{1/2}\na^2(\n u)\in L^\infty(0,T;L^q),\quad \n^{1/2}u_t\in L^{\infty}(0,T;L^2), \\
t\n^{1/2}u_{tt},\quad t\na^2 u_t \in L^{\infty}(0,T;L^2), \\
t\na^3 u \in L^{\infty}(0,T;L^q),\quad t\na u_{tt} \in L^2(0,T;L^2).
\end{cases} 
\ea\ee
\end{theorem}

\begin{theorem}\la{th3}
In addition to the assumptions in Theorem $\ref{th1}$, 
assume further that there exists some point $x_0 \in \rr$ such that $\n_0(x_0)=0$.
Then the unique global strong solution $(\n,u)$ obtained in Theorem \ref{th1} satisfies, for any $r>2$,
\be\la{pbu02}\ba
\lim_{t \to \infty} \| \na \n(\cdot ,t) \|_{L^r} = \infty.
\ea\ee
\end{theorem}

\begin{remark}\la{lrk1}
When $\tilde{\n}=0$, since $(\n,u)$ satisfies \eqref{ssol4} and \eqref{csol3}, similar to the arguments in \cite[Remark 1.1]{LLL},
the solution in Theorem \ref{th2} is in fact a classical solution 
to system \eqref{ns}--\eqref{i3} in $\rr \times (0,\infty)$.

On the other hand, when $\tilde{\n}>0$, it follows from $q>2$, \eqref{ssol5} and \eqref{csol4} that
\be\la{csol5}\ba
( \rho-\tilde{\n},P(\n) -P(\tilde{\n}) ) \in C([0,T];W^{2,q})\hookrightarrow C\left([0,T];C^1(\rr) \right).
\ea\ee
In addition, we deducce from \eqref{ssol5}, \eqref{csol4} 
and the standard embedding that for any $0<\tau<T<\infty$,
\be\la{csol6}\ba
u\in L^\infty(\tau,T;W^{3,q})\cap H^1(\tau ,T;H^2)\hookrightarrow
C\left([\tau ,T];C^2(\rr) \right),
\ea\ee
and	
\be\la{csol7}\ba
u_t\in L^\infty(\tau ,T;H^2)\cap H^1(\tau ,T;H^1)\hookrightarrow
C\left([\tau ,T]; C(\rr) \right).
\ea\ee
Then, by virtue of $(\ref{ns})_1$, \eqref{csol5} and \eqref{csol6}, we have
\be\la{csol8}\ba
\n_t = -\n \div u - u \cdot \na \n \in C(\rr \times [\tau,T]).
\ea\ee

The combination of \eqref{csol5}, \eqref{csol6}, \eqref{csol7} and \eqref{csol8}
shows that the solution in Theorem \ref{th2} is a classical solution to the problem \eqref{ns}--\eqref{i3} in $\rr \times (0,T)$.
\end{remark}

We now make some comments on the analysis of this paper.
For initial data satisfying \eqref{ssol1}, \eqref{csol1}, and \eqref{csol2}, the local existence and uniqueness of classical solutions to (\ref{ns})--(\ref{i3}) follow from arguments similar to those in \cite{LLL}.
To extend the local classical solution, the crucial step is to establish a priori estimates in suitably higher-order norms.
Motivated by \cite{HLX2,HLX1}, the key challenge here lies in deriving an upper bound for the density.

In the case of $\tilde{\n}=0$, we rewrite $(\ref{ns})_1$ as (\ref{2cp61}),
where $G$ represents the effective viscous flux (see (\ref{gw}) for the definition).
This formulation reveals that establishing the upper bound for the density crucially depends on estimating the $L^1(0,T;L^\infty)$ norm of $G$.
According to Gagliardo-Nirenberg's inequality and the elliptic structure of $G$ (see (\ref{p1})), we see that controlling the $L^1(0,T;L^\infty)$ norm of $G$ requires estimating $\| \n \du \|_{L^p(\rr)}$ for $2<p<\infty$.
The analysis presents a technical challenge in unbounded domains due to the failure of Poincar\'e's inequality.
To overcome this, motivated by \cite{LX2}, we first use the $L^1$-integrability of the density to establish a time-independent estimate for the $L^2$-norm of the pressure in both space and time (see (\ref{cp032})).
Based on this crucial estimate, we obtain several decay estimates (see (\ref{sjgj01}) and (\ref{pd11})), which play a key role in deriving the upper bound for the density.
Moreover, we use the weighted inequality (\ref{WPE1}) to show that the $L^p$-norm of $\n \du$ can be bounded by the product of $(1+t)^4$ and some function whose $L^2(0,T)$-norm is independent of time (see (\ref{2cp64})).
With these estimates in hand, we carefully derive the time-uniform upper bound of the density.
Then, following the methods in \cite{HL,LX2}, we establish the necessary higher-order estimates.

In the case of $\tilde{\n}>0$, to establish the global existence, we need to derive the time-uniform estimates.
Compared with the case of $\tilde{\n}=0$, the main difficulty is we have no the $L^1$-integrability of the density, which causes we cannot obtain the time-uniform $L^2(\rr \times (0,T))$ estimate for the pressure term $P-P(\tilde{\n})$.
This prevents us from directly establishing the time-uniform $L^\infty(0,T;L^2(\rr))$ estimate for the gradient of the velocity.
Thus, more refined estimates are required.
First, note that the standard energy estimate (\ref{1cp1}) yields the time-uniform $L^\infty(0,T;L^2(\rr))$ estimate for $P-P(\tilde{\n})$ (see (\ref{1cp27})).
Consequently, by following the proof for the case of $\tilde{\n}=0$, we can obtain a uniform estimate for the gradient of the velocity on the short time interval $(0,\min\{1,T\})$ (see (\ref{1cp02})).
Moreover, by introducing the time layer, we also derive the $\nu$-uniform estimate (\ref{1cp002}).
Then, motivated by \cite{LZ}, using the compensated compactness analysis \cite[Theorem II.1]{CLMS} and the $\nu$-uniform estimate (\ref{1cp002}), we establish the time-uniform $L^\infty(\min\{1,T\},T;L^2)$ estimate for the gradient of the velocity (see (\ref{1cp323})).
This combined with (\ref{1cp002}) gives the desired estimate (\ref{1cp03}).
With this time-uniform estimate for the gradient of the velocity at hand, we follow the argument in the case of $\tilde{\n}=0$ to establish the estimates for the material derivatives of the velocity (see (\ref{1cp04})).
Combining these estimates and proceeding as in \cite{LX3}, we obtain the time-uniform upper bound of the density (see Lemma \ref{cpl5} and its proof).
After establishing the upper bound for the density, we follow the method developed in \cite{HL,LLL,HLX3,HLX2} to derive higher-order derivative estimates for the solution.
These estimates enable us to extend the local solution to a global one.

Finally, we consider the singular limit where solutions of the compressible Navier-Stokes equations converge to those of the inhomogeneous incompressible Navier-Stokes equations.
Following the approach in \cite{DM}, we reformulate $(\ref{ns})_2$ as (\ref{cp33}).
The key step in establishing this convergence is to obtain a $\nu$-uniform bound for $\na G$.
By combining the standard elliptic estimates (\ref{p2}) with the a priori estimates (\ref{cp032}) and (\ref{1cpp2}), we show that the $L^2$-norm of $\na G$ is $\nu$-uniformly bounded.
Notably, our approach does not require the initial velocity field to be divergence-free.
The major technical innovation involves the introduction of the time layer $\si$ when estimating the $L^2$-norm of $\sqrt{\n} \dot{u}$.
Applying the standard compactness arguments, we prove that as the bulk viscosity tends to infinity, the solutions of the compressible Navier-Stokes equations converge to those of the inhomogeneous incompressible Navier-Stokes equations.

The rest of this paper is organized as follows: In Section 2, we present some fundamental inequalities and known facts.
In Sections 3 and 4, we derive some necessary a priori estimates for classical solutions.
Finally, the main results, Theorems \ref{th0}--\ref{th3}, are proved in Section 5.

\section{Preliminaries}
In this section, we will recall some known facts and elementary inequalities
which will be used frequently later.

First, we have the following local existence theory of the classical solution, 
and its proof can be found in \cite{LLL,LZ}.
\begin{lemma}\la{lct}
Assume $\left( \n _0 ,u_0 \right)$ satisfies \eqref{ssol1}, \eqref{csol1} and \eqref{csol2}.
Then there is a small time $T>0$,
such that there exists a unique classical solution $(\n,u)$ to the problem 
\eqref{ns})--\eqref{i3} in $\rr \times (0,T]$ and
when $\tilde{\n}=0$, $(\n,u)$ satisfies \eqref{ssol4} and \eqref{csol3};
when $\tilde{\n}>0$, $(\n,u)$ satisfies \eqref{ssol5} and \eqref{csol4}.
\end{lemma}
Next, the following Gagliardo-Nirenberg's inequalities (see \cite{NI}) will be used frequently later.
\begin{lemma}\la{gn1}
Let $u\in H^1(\rr)$, there exists a positive constant $C$ such that for any $2<p<\infty$
\be\ba\la{gn11}
\| u \|_{L^p} \le Cp^{1/2}\| u \|^{2/p}_{L^2} \| \na u\|^{1-2/p}_{L^2}.
\ea\ee
Furthermore, for $1\le r <\infty$, $2<q<\infty$, 
there exists a positive constant $C$ depending only on $r,\ q $,
such that for every function $v\in L^r(\rr) \cap D^{1,q}(\rr)$ it holds that
\be\ba\la{gn12}
\|v \|_{L^\infty} \le C\|v \|^{r(q-2)/2q+r(q-2)}_{L^r} \|\na v\|^{2q/2q+r(q-2)}_{L^q}.
\ea\ee
\end{lemma}

The following Poincar\'e type inequality can be found in \cite{F}.
\begin{lemma}\la{pt}
Let $v\in H^1(\OM)$, and let $\n$ be a non-negative function satisfying
\be\ba\nonumber
0<M_1\leq \int \n dx,\quad \int \n^\ga dx \leq M_2,
\ea\ee
where $\ga>1$ and $\OM\subset\rr$ is a bounded domain. Then there exists a constant $C$ depending only on 
$M_1,\ M_2$ and $\ga$ such that
\be\ba\la{pt1}
\|v\|_{L^2}^2 \leq C\int \n |v|^2 dx + C \|\na v\|_{L^2}^2.
\ea\ee
\end{lemma}

The following weighted $L^p$ estimates can be found in \cite[Theorem B.1]{L1}.
\begin{lemma}\la{WPE}
For $m \in [2,\infty)$ and $\theta \in (1+\frac{m}{2},\infty)$ 
there exists a positive generic constant $C$, 
such that for any $v \in D^1(\rr)$, we have
\be\la{WPE1}\ba
\left( \int _{{\mathbb {R}^2 }} \frac{|v|^m}{e+|x|^2}(\log (e+|x|^2))^{-\theta }dx \right) ^{1/m}
\le C\Vert v\Vert _{L^2(B_1)}+C\Vert \nabla v\Vert _{L^2({\mathbb {R}^2 }) }.
\ea\ee
\end{lemma}

Subsequently, the following estimate of $\| \n v\|_{L^r}$
plays an important role in the estimates when $\tilde{\rho}=0$,
and the proof can be found in \cite[Lemma 2.4]{LX2}.
\begin{lemma}\la{esnr}
For $N_*\ge 1$ and positive constants $M_3$, $M_4$, $\beta$, we suppose that $\n$ satisfies
\be\la{esnr1}\ba
0\le \rho \le M_3, \quad M_4\le \int _{B_{N_*}}\rho dx ,\quad{\bar{x}}^\beta \rho \in L^1(\rr).
\ea\ee

Then for any $r \in [2,\infty)$, there exists a positive constant $C$ depending only on $M_3$, $M_4$, $\beta$ and $r$, such that for every 
$v\in \left. \left\{ v\in D^1 ({\rr})\right| \rho^{1/2}v\in L^2(\rr) \right\}$
\be\la{esnr2}\ba
\left( \int _{\mathbb {R}^2}\rho |v |^r dx\right) ^{1/r} 
\le C N_*^3 (1+\Vert {\bar{x}}^\beta \rho \Vert _{L^1(\rr)}) 
\left( \Vert \sqrt{\n} v\Vert _{L^2(\mathbb {R}^2)} 
+ \Vert \nabla v \Vert _{L^2(\mathbb {R}^2)}\right).
\ea\ee
\end{lemma}

The following estimate of $\|  u\|_{L^2}$
is of great significance in the estimations when $\tilde{\rho}>0$,
and the proof can be found in \cite[Lemma 2.5]{LZ}.
\begin{lemma}\la{LZ1}
If $\tilde{\rho}>0$ and $(\n,v)$ satisfies
\be\la{LZ01}\ba
\n-\tilde{\rho} \in L^2, \quad v \in D^1, \quad \sqrt{\n}v \in L^2.
\ea\ee
Then, there exists a positive constant $C$ depending only on $\tilde{\rho}$ such that
\be\la{LZ02}\ba
\| v \|^2_{L^2} \le C \left( \Vert \sqrt{\rho} v \Vert^2_{L^2}+\| \n-\tilde{\rho} \|^2_{L^2} \| \na v \|^2_{L^2} \right).
\ea\ee
\end{lemma}

Next, for $\na^{\bot}:=(-\pa_2,\pa_1)$, denoting the material derivative 
of $f$ by $\frac{D}{Dt}f=\dot{f}:=f_t + u\cdot\na f$,
we now state standard $L^{p}$-estimate for the following elliptic system derived from the momentum equations in (\ref{ns}): 
\be\ba\la{p1}
\Delta G = \div(\n\dot{u}),\quad \mu \Delta \o = \na^{\bot} \cdot(\n \dot{u}),
\ea\ee
with
\be\ba\la{gw}
G:=(2\mu + \lam)\div u - ( P-P(\tilde{\n}) ),\quad \o :=\pa_1 u^2 - \pa_2 u^1.
\ea\ee

With these notations defined above, we state the following lemma.
\begin{lemma}\la{estg}
Let $(\n,u)$ be a smooth solution of \eqref{ns}. Then for $1<p<\infty$ and positive integer $k \ge 1$, 
there exists a positive constant $C$ depending only on $k,\ p,\ \mu$, such that 	
\be\la{p2}
\|\na^k G\|_{L^p}+\|\na^k \o\|_{L^p} \leq C \| \na ^{k-1}(\n \dot{u}) \|_{L^p}.
\ee
\end{lemma}

Next, the following div-curl estimate will be frequently used in later
arguments.
\begin{lemma}\la{dc}
Let $k \ge 1$ be a positive integer and $1<p<\infty$.
Then there exists a positive constant $C$ depending only on $k,\ p$, 
 such that for every $\na u\in W^{k,p}$  it holds that 
\be\ba\la{dc1}
\|\na u\|_{W^{k,p}} \le C\left( \|\div u\|_{W^{k,p}} +\| \o\|_{W^{k,p}} \right).
\ea\ee
\end{lemma}

To estimate $\| \na u\|_{L^{\infty}}$ and $\| \na \n\|_{L^{q}}$ 
we require the following Beale-Kato-Majda type inequality, 
which was established in \cite{K} when $\div u \equiv 0$. 
For further reference, we direct readers to \cite{BKM,HLX1}.
\begin{lemma}\la{bkm}
For $2<q<\infty$, 
there exists a positive constant $C$ 
depending only on $q $ such that for every function $\na u\in W^{1,q}$, it holds that
\be\ba\la{bkm1}
 \|\na u\|_{L^\infty} \le C \left( \|\div u\|_{L^\infty}+ \|\o\|_{L^\infty} \right)\log \left(e+ \|\na^2 u\|_{L^q} \right)+ C\|\na u\|_{L^2}+C.
\ea\ee
\end{lemma}

\section{A Priori Estimates (\uppercase\expandafter{\romannumeral1}): Lower Order Estimates}
In this section, we establish some necessary a priori estimates.
Let $(\n,u)$ be a classical solution to (\ref{ns})--(\ref{i3}) on $\rr \times (0,T]$ obtained by Lemma \ref{lct}.

We set
\be\la{defa1}\ba
A_1^2(t)\triangleq \int \left( \mu \o^2(t)+\frac{G^2(t)}{2\mu+\lambda} \right) dx,
\ea\ee
and
\be\la{defa2}\ba
A_2^2(t)\triangleq\int \rho(t)|\dot{u}(t)|^2 dx.
\ea\ee

\subsection{Far-field density is vacuum}
In this subsection, we assume that $\tilde{\n}=0$, and that the initial data $(\n_0,u_0)$ satisfy (\ref{wsol1}) and $(\ref{wsol01})_1$ with $\n_0>0$.

The primary aim of this subsection is to derive the following a priori estimates.
\begin{proposition}\la{l5}
There exists a positive constant $\nu_1$ depending only on
$N_0$, $\ga$, $\mu$, $a$, $E_0$, $\|{\bar{x}}^a \n_0\|_{L^1}$, $\|\n_0\|_{L^\infty}$, and $\| \na u_0\|_{L^2}$
such that if $(\n,u)$ satisfies
\be\ba\la{cp051}
\sup_{0\leq t\leq T} \|\n\|_{L^\infty} 
\leq 2 \left( 1 + \| \n_0 \|_{L^\infty} \right),
\ea\ee
then
\be\ba\la{cp052}
\sup_{0\leq t\leq T} \|\n\|_{L^\infty} 
\leq \frac{3}{2} \left( 1 + \| \n_0 \|_{L^\infty} \right),
\ea\ee
provided $\nu \geq \nu_1$.
\end{proposition}

The proof of Proposition \ref{l5} will be postponed to the end of this subsection.

We first state the standard energy estimate.
\begin{lemma}\la{l1}
There exists a positive constant $C$ depending only on
$\ga$, $\mu$, and $E_0$ such that
\be\ba\la{cp1}
\sup_{0\leq t\leq T} \int \left( \rho |u|^2 + \n^\ga \right) dx
+ \int_0^T \int \left( \mu |\na u|^2 + \nu (\div u)^2 \right) dx dt \leq C.
\ea\ee
\end{lemma}
\begin{proof}
Multiplying $(\ref{ns})_2$ by $u$, integrating the resulting equation over $\rr$, and using $(\ref{ns})_1$, we obtain (\ref{cp1}).
\end{proof}

\begin{lemma}\la{l3}
Let $(\n,u)$ be a classical solution to \eqref{ns}--\eqref{i3} satisfying \eqref{cp051}.
Then there exists a positive constants $C$ depending only on
$\ga$, $\mu$, $E_0$, and $\| \n_0 \|_{L^\infty}$ such that
\be\la{cp031}\ba
\sup_{0\le t\le T} A^2_1 + \int_0^T \left( \| \sqrt{\n} \dot{u}\|^2_{L^2} + \frac{1}{\nu} \| P \|^2_{L^2} \right) dt
\le C( 1 + A^2_1(0)),
\ea\ee
and
\be\la{cp032}\ba
\sup_{0\le t\le T} \si A^2_1 + \int_0^T \si \left( \| \sqrt{\n} \dot{u}\|^2_{L^2} + \frac{1}{\nu} \| P \|^2_{L^2} \right) dt
\le C,
\ea\ee
with
\be\nonumber
\si(t) \triangleq \min\{1,t\}.
\ee
\end{lemma}

\begin{proof}
First, from $(\ref{ns})_2$ and (\ref{gw}), we deduce that
\be\la{cp21}\ba
P=(-\Delta )^{-1} \mathrm{div } (\rho {\dot{u}} )+\nu \mathrm{div } u.
\ea\ee
Multiplying (\ref{cp21}) by $P$ and integrating over $\rr$, after using Sobolev's and H\"older's inequalities, we derive
\be\nonumber\ba
\int P^2dx & \le \Vert (-\Delta )^{-1} \mathrm{div } (\rho {\dot{u}} )\Vert _{L^{4\gamma }} \Vert P\Vert _{L^{4\gamma /(4\gamma -1)}} +\nu \| \div u\|_{L^2} \Vert P\Vert _{L^2} \\ 
& \le C\Vert \rho {\dot{u}} \Vert _{L^{4\gamma /(2\gamma +1)}}\Vert \rho \Vert _{L^1}^{1/2}\Vert \rho \Vert _{L^{2\gamma }}^{\gamma -1/2}+\nu \| \div u\|_{L^2} \Vert P\Vert _{L^2} \\ 
& \le C \Vert \sqrt{\n}\Vert _{L^{4\gamma }} \Vert \sqrt{\n} {\dot{u}} \Vert _{L^2}\Vert \rho \Vert _{L^1}^{1/2}\Vert \rho \Vert _{L^{2\gamma }}^{\gamma -1/2} +\nu \| \div u\|_{L^2} \Vert P\Vert _{L^2} \\ 
& \le C\Vert P\Vert _{L^2} \Vert \sqrt{\n} {\dot{u}} \Vert _{L^2} +\nu \| \div u\|_{L^2} \Vert P\Vert _{L^2} \\
& \le \frac{1}{2} \Vert P\Vert _{L^2}+C \| \sqrt{\n} \dot{u}\|^2_{L^2}+ \nu^2 \| \div u\|^2_{L^2},
\ea\ee
which gives
\be\ba\la{cp22}
\| P \|^2_{L^2} \le C \| \sqrt{\n} \dot{u}\|^2_{L^2} + 2 \nu^2 \| \div u\|^2_{L^2}.
\ea\ee
Direct calculations show that 	
\be\la{cp31}\ba
\na^{\bot}\cdot \dot u= \frac{D}{Dt}\o +(\p_1u\cdot\na) u^2
-(\p_2u\cdot\na)u^1  = \frac{D}{Dt}\o + \o \div u , 
\ea\ee
and that
\be\la{cp32}\ba 
\div  \dot u&=\frac{D}{Dt}\div u +(\p_1u\cdot\na) u^1+(\p_2u\cdot\na)u^2\\&
=\frac{1}{\nu} \frac{D}{Dt}G+ \frac{1}{\nu} \frac{D}{Dt}( P-P( \tilde{\n} ) )
+ 2\nabla u^1\cdot\nabla^{\perp}u^2 + (\div u)^2.	
\ea\ee

Then, we rewrite $(\ref{ns})_2$ as 
\be\la{cp33}\ba
\n\dot{u} = \na G + \mu\na^{\bot}\o.
\ea\ee
Multiplying both sides of (\ref{cp33}) by $2 \dot u$ and then integrating
the resulting equality over $\rr$ lead to
\be\la{cp34}\ba
& \frac{d}{dt} \int \left(\mu \o^2 + \frac{G^2}{\nu}\right)dx + 2\| \sqrt{\n} \dot{u}\|^2_{L^2}\\
& = -\mu \int \o^2\div udx - 4\int G\nabla u^1\cdot\nabla^{\perp}u^2dx- 2\int G(\div u)^2dx\\
&\quad +\frac{1}{\nu} \int G^2\div udx +\frac{2\ga}{\nu} \int P G\div udx= \sum_{i=1}^5I_i,
\ea\ee
where we have used (\ref{cp31}) and (\ref{cp32}). Next, we estimate each $I_i$ as follows:

First, combining (\ref{gn11}), (\ref{p2}) and H\"older's inequality yields
\be\la{cp35}\ba
|I_1| &\le C \| \o\|^2_{L^4} \| \div u\|_{L^2} \\
& \le C \| \o\|_{L^2} \| \na \o\|_{L^2} \| \div u\|_{L^2} \\
& \le C \| \sqrt{\n} \dot{u}\|_{L^2} \| \o\|_{L^2} \| \div u\|_{L^2} \\
& \le \frac{1}{16} \| \sqrt{\n} \dot{u}\|^2_{L^2} + C \| \na u\|^4_{L^2}.
\ea\ee
Next, we will use the idea in \cite{HL2} to estimate $I_2$. Observe that 
\be\la{cp36}\ba
\na^{\bot} \cdot \na u^1=0,\quad \div \na^{\bot}u^2=0,
\ea\ee
we can infer from \cite[Theorem II.1]{CLMS} that
\be\la{cp37}\ba
\|\na u^1\cdot \na^{\bot}u^2\|_{\mathcal{H}^1}\le C\|\na u\|_{L^2}^2.
\ea\ee
Based on the fact that $\mathcal{BMO}$ is the dual space of $\mathcal{H}^1$ (see \cite{FC}), we obtain
\be\la{cp38}\ba
|I_2| &\le C \| G\|_{\mathcal{BMO}} \|\na u^1\cdot \na^{\bot}u^2\|_{\mathcal{H}^1} \\
& \le C \| \na G\|_{L^2} \| \na u\|^2_{L^2} \\
& \le C \| \sqrt{\n} \dot{u}\|_{L^2} \| \na u\|^2_{L^2} \\
& \le \frac{1}{16} \| \sqrt{\n} \dot{u}\|^2_{L^2}+ C \| \na u\|^4_{L^2}.
\ea\ee

It follows from (\ref{gn11}), (\ref{p2}), (\ref{gw}), (\ref{cp22}), and H\"older's inequality that
\be\la{cp311}\ba
\sum_{i=3}^5I_i & \le \frac{C}{\nu} \int G^2 |\div u| dx + \frac{C}{\nu}\int P |G| |\div u| dx \\
& \le \frac{C}{\nu} \| G \|^2_{L^4} \| \div u \|_{L^2} + \frac{C}{\nu} \| G \|_{L^2} \| \div u \|_{L^2} \\
& \le \frac{C}{\nu} \| G \|_{L^2} \| \na G \|_{L^2} \| \div u \|_{L^2}
+ C \| \div u \|^2_{L^2} + \frac{C}{\nu} \| P \|_{L^2} \| \div u \|_{L^2} \\
& \le \frac{C}{\nu} \| G \|_{L^2} \| \sqrt{\n} \dot{u} \|_{L^2} \| \div u \|_{L^2}
+ C \| \div u \|^2_{L^2} + \frac{C}{\nu} \| \sqrt{\n} \dot{u} \|_{L^2} \| \div u \|_{L^2} \\
& \le \frac{1}{16} \| \sqrt{\n} \dot{u}\|^2_{L^2}
+ C A^2_1 \| \na u\|^2_{L^2}+C\| \na u\|^2_{L^2}.
\ea\ee
Putting (\ref{cp35}), (\ref{cp38}) and (\ref{cp311}) into (\ref{cp34}) implies that
\be\la{cp312}\ba
\frac{d}{dt} A^2_1 + \| \sqrt{\n} \dot{u}\|^2_{L^2}
\le C A^2_1 \| \na u\|^2_{L^2}+C\| \na u\|^2_{L^2}.
\ea\ee
where we have used the following fact:
\be\la{cp39}\ba
\nu \| \div u \|^2_{L^2} + \| \na u\|^2_{L^2} & \le C \left( \nu \| \div u\|^2_{L^2}+\| \o \|^2_{L^2} \right) \\
& \le CA^2_1+\frac{C}{\nu} \left( \| G \|^2_{L^2}+\| P \|^2_{L^2} \right) \\
& \le C( A^2_1 + \frac{1}{\nu} \| P \|^2_{L^2} ),
\ea\ee
due to (\ref{dc1}) and (\ref{cp051}).

In addition, multiplying (\ref{cp312}) by $\si$ results in
\be\la{cp313}\ba
\frac{d}{dt} (\si A^2_1) + \si \| \sqrt{\n} \dot{u}\|^2_{L^2}
\le \si' A^2_1 +C \si A^2_1 \| \na u\|^2_{L^2} + C \| \na u\|^2_{L^2}.
\ea\ee
Consequently, by using (\ref{cp1}) and Gr\"onwall's inequality,
we derive (\ref{cp031}) and (\ref{cp032}).
This completes the proof of Lemma \ref{l3}.
\end{proof}

\begin{lemma}\la{sjgjl}
Let $(\n,u)$ be a classical solution to \eqref{ns}--\eqref{i3} satisfying \eqref{cp051}.
Then there exists a positive constants $C$ depending only on
$\ga$, $\mu$, $E_0$, and $\| \n_0 \|_{L^\infty}$ such that
\be\la{sjgj01}\ba
\sup_{\si(T) \le t\le T} t \left( \| \na u \|^2_{L^2} + \nu \| \div u \|^2_{L^2} + \frac{1}{\nu} \| P \|^2_{L^2} \right)
+ \int_{\si(T)}^T t \left( \| \sqrt{\n} \dot{u}\|^2_{L^2} + \frac{1}{\nu^2} \| P \|^3_{L^3} \right) dt
\le C,
\ea\ee
and
\be\la{sjgj02}\ba
\sup_{\si(T) \le t\le T} \left( \frac{1}{\nu^2} t^2 \| P \|_{L^3}^3 \right) + \int_{\si(T)}^T \frac{1}{\nu^3} t^2 \| P \|^4_{L^4} dt
\le C.
\ea\ee
\end{lemma}
\begin{proof}
It follows from (\ref{cp34}), (\ref{cp35}), (\ref{cp38}), (\ref{cp22}), (\ref{gw}), and Young's inequality that
\be\la{sjgj1}\ba
\frac{d}{dt} A^2_1 + \frac{3}{2} \| \sqrt{\n} \dot{u} \|^2_{L^2}
& \le C \| \na u\|^4_{L^2}
+ \frac{C}{\nu} \int G^2 |\div u| dx + \frac{C}{\nu}\int P |G| |\div u| dx \\
& \le C \| \na u\|^4_{L^2} + \frac{C}{\nu} \| G \|^2_{L^3} \| \div u \|_{L^3} + \frac{C}{\nu} \| P \|_{L^3} \| G \|_{L^3} \| \div u \|_{L^3} \\
& \le C \| \na u\|^4_{L^2} + \frac{C}{\nu^2} \| G \|^3_{L^3} + \frac{2\ga -1}{4\nu^2} \| P \|^3_{L^3} \\
& \le C \| \na u\|^4_{L^2} + \frac{C}{\nu^2} \| G \|^2_{L^2} \| \na G \|_{L^2} + \frac{2\ga -1}{4\nu^2} \| P \|^3_{L^3} \\
& \le C \| \na u\|^4_{L^2} + \frac{C}{\nu^2} \| G \|^2_{L^2} \| \sqrt{\n} \dot{u} \|_{L^2} + \frac{2\ga -1}{4\nu^2} \| P \|^3_{L^3} \\
& \le \frac{1}{4} \| \sqrt{\n} \dot{u}\|^2_{L^2} + C \| \na u\|^4_{L^2}
+ \frac{C}{\nu^4} \| G \|^4_{L^2} + \frac{2\ga -1}{4\nu^2} \| P \|^3_{L^3},
\ea\ee
which yields
\be\la{sjgj2}\ba
\frac{d}{dt} A^2_1 + \frac{5}{4} \| \sqrt{\n} \dot{u}\|^2_{L^2}
\le C \| \na u\|^4_{L^2} + \frac{C}{\nu^2} \| G \|^4_{L^2} + \frac{2\ga -1}{4\nu^2} \| P \|^3_{L^3}.
\ea\ee
Since $P$ satisfies
\be\la{sjgj3}\ba
P_t + u \cdot \na P + \ga P \div u = 0.
\ea\ee
For any $2 \le p < \infty$, multiplying (\ref{sjgj3}) by $p P^{p-1}$, integrating by parts over $\rr$, and using (\ref{gw}), (\ref{gn11}), and Young's inequality, we arrive at
\be\la{sjgj4}\ba
\frac{d}{dt} \| P \|^p_{L^p} + \frac{p\ga -1}{\nu} \| P \|^{p+1}_{L^{p+1}}
& = - \frac{p\ga -1}{\nu} \int P^p G dx \\
& \le \frac{p\ga -1}{2\nu} \| P \|^{p+1}_{L^{p+1}} + \frac{C}{\nu} \| G \|^{p+1}_{L^{p+1}} \\
& \le \frac{p\ga -1}{2\nu} \| P \|^{p+1}_{L^{p+1}}
+ \frac{C}{\nu} \| G \|^{2}_{L^2} \| \sqrt{\n} \dot{u} \|^{p-1}_{L^2},
\ea\ee
which shows that
\be\la{sjgj5}\ba
\frac{d}{dt} \left( \frac{1}{\nu} \| P \|^p_{L^p} \right) + \frac{p\ga -1}{2 \nu^2} \| P \|^{p+1}_{L^{p+1}}
\le \frac{C}{\nu^2} \| G \|^{2}_{L^2} \| \sqrt{\n} \dot{u} \|^{p-1}_{L^2}.
\ea\ee
Choosing $p=2$ in (\ref{sjgj5}) gives
\be\la{sjgj6}\ba
\frac{d}{dt} \left( \frac{1}{\nu} \| P \|^2_{L^2} \right) + \frac{2\ga -1}{2 \nu^2} \| P \|^{3}_{L^{3}}
\le \frac{C}{\nu^2} \| G \|^{2}_{L^2} \| \sqrt{\n} \dot{u} \|_{L^2}
\le \frac{1}{4} \| \sqrt{\n} \dot{u} \|^2_{L^2} + \frac{C}{\nu^4} \| G \|^4_{L^2}.
\ea\ee
Summing (\ref{sjgj2}) and (\ref{sjgj5}) and using (\ref{gw}) and (\ref{cp39}) lead to
\be\la{sjgj7}\ba
& \frac{d}{dt} \left( A^2_1 + \frac{1}{\nu} \| P \|_{L^2}^2 \right) + \| \sqrt{\n} \dot{u}\|^2_{L^2} + \frac{2\ga -1}{4\nu^2} \| P \|^3_{L^3} \\
& \le C \| \na u\|^4_{L^2} + \frac{C}{\nu^2} \| G \|^4_{L^2} \\
& \le C \left( A^2_1 + \frac{1}{\nu} \| P \|_{L^2}^2 \right)
\left( \| \na u\|^2_{L^2} + \nu \| \div u \|^2_{L^2} + \frac{C}{\nu} \| P \|^2_{L^2} \right).
\ea\ee
Multiplying (\ref{sjgj7}) by $t$ yields
\be\la{sjgj8}\ba
& \frac{d}{dt} \left( t \left( A^2_1 + \frac{1}{\nu} \| P \|_{L^2}^2 \right) \right) + t \| \sqrt{\n} \dot{u}\|^2_{L^2} + \frac{2\ga -1}{4\nu^2} t \| P \|^3_{L^3} \\
& \le \left( A^2_1 + \frac{1}{\nu} \| P \|_{L^2}^2 \right)
+ C t \left( A^2_1 + \frac{1}{\nu} \| P \|_{L^2}^2 \right)
\left( \| \na u\|^2_{L^2} + \nu \| \div u \|^2_{L^2} + \frac{C}{\nu} \| P \|^2_{L^2} \right).
\ea\ee
Applying Gr\"onwall's inequality to (\ref{sjgj8}) and using (\ref{cp1}) and (\ref{cp032}), we obtain
\be\la{sjgj9}\ba
\sup_{\si(T) \le t\le T} t \left( A^2_1 + \frac{1}{\nu} \| P \|_{L^2}^2 \right) + \int_{\si(T)}^T t \left( \| \sqrt{\n} \dot{u}\|^2_{L^2} + \frac{1}{\nu^2} \| P \|^3_{L^3} \right) dt
\le C,
\ea\ee
which together with (\ref{cp39}) gives (\ref{sjgj01}).

Moreover, choosing $p=3$ in (\ref{sjgj4}), we have
\be\la{sjgj10}\ba
\frac{d}{dt} \| P \|^3_{L^3} + \frac{3\ga -1}{2 \nu} \| P \|^{4}_{L^{4}}
\le \frac{C}{\nu} \| G \|^{2}_{L^2} \| \sqrt{\n} \dot{u} \|^{2}_{L^2},
\ea\ee
which implies that
\be\la{sjgj11}\ba
\frac{d}{dt} \left( \frac{1}{\nu^2} t^2 \| P \|^3_{L^3} \right) + \frac{3\ga -1}{2 \nu^3} t^2 \| P \|^{4}_{L^{4}}
\le \frac{2}{\nu^2} t \| P \|^3_{L^3} + \frac{C}{\nu^3} t^2 \| G \|^{2}_{L^2} \| \sqrt{\n} \dot{u} \|^{2}_{L^2}.
\ea\ee
Integrating (\ref{sjgj11}) over $(\si(T),T)$ and using (\ref{cp031}), (\ref{cp032}), and (\ref{sjgj9}), we arrive at (\ref{sjgj02}) and complete the proof of Lemma \ref{sjgjl}.
\end{proof}

\begin{lemma}\la{g1}
Let $(\n,u)$ be a classical solution to \eqref{ns}--\eqref{i3} satisfying \eqref{cp051}.
Then there exists a positive constants $C$ depending only on
$\ga$, $\mu$, $E_0$, and $\| \n_0 \|_{L^\infty}$ such that
\be\ba\la{pd11}
\sup_{0\le t\le T}
t^2 \int\n|\dot u|^2dx
+\int_0^{T} t^2 \| \na\dot u\|^2_{L^2} dt \le C.
\ea\ee
Moreover, for any $p \in [2,\infty)$, there exists a positive constants $C$ depending only on
$p$, $\ga$, $\mu$, $E_0$, and $\| \n_0 \|_{L^\infty}$ such that
\be\la{pd011}\ba
\sup_{\si(T) \le t \le T} t^{p-1} \left( \frac{1}{\nu^{p-1}} \| P \|^{p}_{L^{p}} + \| \na u \|^p_{L^p} \right)
\le C.
\ea\ee
\end{lemma}
\begin{proof}
First, using (\ref{gw}), we rewrite $(\ref{ns})_2$ as follows:
\be\la{pd12}\ba
\n \dot{u} = \mu \Delta u+\frac{\nu-\mu}{\nu} \na G -\frac{\mu}{\nu} \na P.
\ea\ee
Adapting the method of \cite{H1,LZ2}, we apply the operator 
$ \dot u^j[\pa/\pa t+\div(u\cdot)]$
to $ (\ref{pd12})^j,$ sum over $j,$ and integrate over $\rr$, which yields
\be\la{pd13}\ba &
\left(\frac{1}{2}\int\rho|\dot{u}|^2dx \right)_t\\
& = \mu\int\dot{u}^j \left[\Delta u_t^j + \text{div}(u\Delta u^j) \right] dx 
-\frac{\mu}{\nu}\int\dot{u}^j \left[\p_jP_t+\text{div}(u \p_j (P-P(\tilde{\n})) \right]dx \\
& \quad + \frac{\nu-\mu}{\nu} \int \dot{u}^j \left[\p_t\p_j G+\text{div}(u\p_j G)\right] dx \\
& \triangleq\sum_{i=1}^{3}N_i. \ea\ee
Integrating by parts and applying Young's inequality, we derive
\be\la{pd14}\ba
N_1 & = \mu\int \dot{u}^j \left[\Delta u_t^j
+ \text{div}(u\Delta u^j) \right]dx \\
& = - \mu\int \left[|\nabla\dot{u}|^2 + \p_i\dot{u}^j\p_ku^k\p_iu^j - \p_i\dot{u}^j\p_iu^k\p_ku^j - \p_k\dot{u}^j \p_iu^j\p_iu^k \right]dx \\
&\le -\frac{ 3\mu}{4} \| \na \dot{u} \|_{L^2}^2 + C \| \na u \|_{L^4}^4.
\ea\ee
Similarly, by virtue of $(\ref{ns})_1$ and Young's inequality, it holds that
\be\la{pd15}\ba
N_2 & = - \frac{\mu}{\nu} \int \dot{u}^j \left[\p_j P_t + \text{div}( u \p_j P ) \right] dx \\
& = \frac{\mu}{\nu} \int \left[-P^{'}\rho\div \dot u\div u + P \div \dot u\div u 
- P \p_i\dot{u}^j\p_ju^i \right] dx \\
& \le \frac{\mu}{8} \| \na \dot{u} \|_{L^2}^2 + C \| \na u \|^4_{L^4} + \frac{C}{\nu^4} \| P \|^4_{L^4}.
\ea\ee
Then, integration by parts combined with (\ref{p2}), (\ref{cp37}) and Young's inequality yields
\be\la{pd16}\ba
N_3 & = \frac{\nu-\mu}{\nu} \int\dot{u}^j \left[\p_j \p_t G + \text{div}( u \p_j G ) \right] dx \\
& = -\frac{\nu-\mu}{\nu} \int \div \dot{u} \left( \dot{G}-u \cdot \na G \right) dx
+\frac{\nu-\mu}{\nu} \int \dot{u}^j \p_j G \div u + \dot{u}^j u \cdot \na \p_j G dx \\
& = -\frac{\nu-\mu}{\nu} \int \div \dot{u} \dot{G} dx
+\frac{\nu-\mu}{\nu} \int -G \div \dot{u} \div u - G \dot{u}^j \p_j \div u
- \dot{u}^j \p_j u \cdot \na G dx \\
& = -\frac{\nu-\mu}{\nu} \int \div \dot{u} \dot{G} dx
+\frac{\nu-\mu}{\nu} \int -G \div \dot{u} \div u + G \p_j u \cdot \na \dot{u}^j dx \\
& = -\frac{\nu-\mu}{\nu} \int \div \dot{u} \dot{G} dx
+\frac{\nu-\mu}{\nu} \int G \left( \na u^1 \cdot \na^\bot \dot{u}^2 - \na u^2 \cdot \na^\bot \dot{u}^1 \right) dx \\
& \le -\frac{\nu-\mu}{\nu} \int \div \dot{u} \dot{G} dx + C \| \na G \|_{L^2} \| \na u \|_{L^2} \| \na \dot{u} \|_{L^2} \\
& \le -\frac{\nu-\mu}{\nu} \int \div \dot{u} \dot{G} dx + \frac{\mu}{8} \| \na \dot{u} \|_{L^2}^{2}
+ C \| \sqrt{\n} \dot{u} \|^2_{L^2} \| \na u \|^2_{L^2},
\ea\ee
where we have used the fact that
\be\la{pd17}\ba
\sum_{j=1}^{2} \p_j u \cdot \na \dot{u}^j
= \div u \div \dot{u} + \na u^1 \cdot \na^\bot \dot{u}^2 - \na u^2 \cdot \na^\bot \dot{u}^1.
\ea\ee
For the first term in the last line of (\ref{pd16}),
by the definition of $G$ and $(\ref{ns})_1$, we have
\be\la{pd18}\ba
\div \dot{u} & = (\div u)_t + \p_i u^j \p_j u^i + u \cdot \na \div u \\
& = \frac{1}{\nu} \dot{G} + \frac{1}{\nu} (P_t+u \cdot \na P) + \p_i u^j \p_j u^i \\
& = \frac{1}{\nu} \dot{G} - \frac{1}{\nu} \ga P \div u + \p_i u^j \p_j u^i,
\ea\ee
which together with Young's inequality leads to
\be\la{pd19}\ba
-\frac{\nu-\mu}{\nu} \int \div \dot{u} \dot{G} dx
& = -\frac{\nu-\mu}{\nu^2} \| \dot{G} \|^2_{L^2} 
-\frac{\nu-\mu}{\nu} \int \dot{G} \p_i u^j \p_j u^i dx 
+ \frac{\nu-\mu}{\nu^2} \ga \int \dot{G} P \div u dx \\
& \le -\frac{\nu-\mu}{2\nu^2} \| \dot{G} \|^2_{L^2} + \frac{C}{\nu^3} \| P \|^4_{L^4} + \frac{C}{\nu^3} \| G \|^4_{L^4} + C \| \na u \|^4_{L^4}
- \frac{\nu-\mu}{\nu} \int \dot{G} \p_i u^j \p_j u^i dx.
\ea\ee
For the last term in the final line of (\ref{pd19}), integration by parts implies
\be\la{pd110}\ba
\int \dot{G} \p_i u^j \p_j u^i dx 
& = \int ( G_t + u \cdot \na G ) \p_i u \cdot \na u^i dx \\
& = \frac{d}{dt} \left( \int G \p_i u \cdot \na u^i dx \right) 
- 2 \int G \p_i u \cdot \na u^i_t dx \\
& \quad -\int G \div u \p_i u \cdot \na u^i dx 
-2 \int G u \cdot \na \p_i u \cdot \na u^i dx \\
& = \frac{d}{dt} \left( \int G \p_i u \cdot \na u^i dx \right) 
- 2 \int G \p_i u \cdot \na \dot{u}^{i} dx \\
& \quad +2 \int G \p_i u \cdot \na u \cdot \na u^i dx
-\int G \div u \p_i u \cdot \na u^i dx,
\ea\ee
which yields that
\be\la{pd111}\ba
-\frac{\nu-\mu}{\nu} \int \dot{G} \p_i u^j \p_j u^i dx
& =-\frac{\nu-\mu}{\nu} \frac{d}{dt} \left( \int G \p_i u \cdot \na u^i dx \right)
+ \frac{2(\nu-\mu)}{\nu} \int G \p_i u \cdot \na \dot{u}^{i}dx \\
& \quad - \frac{2(\nu-\mu)}{\nu} \int G \p_i u \cdot \na u \cdot \na u^i dx
+\frac{\nu-\mu}{\nu} \int G \div u \p_i u \cdot \na u^i dx.
\ea\ee
We now estimate the last three terms on the right-hand side of (\ref{pd111}) term by term.
Applying (\ref{pd17}), (\ref{gw}), (\ref{p2}) and Young's inequality, we obatin
\be\la{pd112}\ba
\frac{2(\nu-\mu)}{\nu} \int G \p_i u \cdot \na \dot{u}^{i} dx
& \le 2 \left| \int G \left( \div u \div \dot{u} + \na u^1 \cdot \na^\bot \dot{u}^2 - \na u^2 \cdot \na^\bot \dot{u}^1 \right) dx \right| \\
& \le C \| \na \dot{u} \|_{L^2} \| G \|_{L^4} \| \div u \|_{L^4}
+ C \| \na G \|_{L^2} \| \na \dot{u} \|_{L^2} \| \na u \|_{L^2} \\
& \le \frac{\mu}{8} \| \na \dot{u} \|^2_{L^2} + \frac{C}{\nu} \| G \|^4_{L^4} + \frac{C}{\nu^3} \| P \|^4_{L^4}
+ C \| \sqrt{\n} \dot{u} \|^2_{L^2} \| \na u \|^2_{L^2}.
\ea\ee
Then, we note that the following equality:
\be\la{pd114}\ba
\sum_{i=1}^{2} \p_i u \cdot \na u \cdot \na u^i
= (\div u)^3 +3\div u \na u^1 \cdot \na^\bot u^2,
\ea\ee
which together with (\ref{gw}) and Young's inequality leads to
\be\la{pd113}\ba
- \frac{2(\nu-\mu)}{\nu} \int G \p_i u \cdot \na u \cdot \na u^i dx
& \le 2 \left| \int G \left( (\div u)^3+3\div u \na u^1 \cdot \na^\bot u^2 \right) dx \right| \\
& \le C \| G \|_{L^4} \| \div u \|_{L^4} \| \na u \|^2_{L^4} \\
& \le \frac{C}{\nu} \| G \|^4_{L^4} + \frac{C}{\nu^3} \| P \|^4_{L^4} 
+ C \| \na u \|^4_{L^4}.
\ea\ee
Similarly,
\be\la{pd115}\ba
\frac{\nu-\mu}{\nu} \int G \div u \p_i u \cdot \na u^i dx 
\le \frac{C}{\nu} \| G \|^4_{L^4} + \frac{C}{\nu^3} \| P \|^4_{L^4} 
+ C \| \na u \|^4_{L^4}.
\ea\ee
Putting (\ref{pd112}), (\ref{pd113}) and (\ref{pd115}) into (\ref{pd111}), we infer that
\be\la{pd116}\ba
-\frac{\nu-\mu}{\nu} \int \dot{G} \p_i u^j \p_j u^i dx
& \le -\frac{\nu-\mu}{\nu} \frac{d}{dt} \left( \int G \p_i u \cdot \na u^i dx \right)
+ \frac{\mu}{8} \| \na \dot{u} \|^2_{L^2} + \frac{C}{\nu} \| G \|^4_{L^4} \\
& \quad + \frac{C}{\nu^3} \| P \|^4_{L^4}
+ C \| \sqrt{\n} \dot{u} \|^2_{L^2} \| \na u \|^2_{L^2} + C \| \na u \|^4_{L^4}.
\ea\ee
Therefore, it follows from (\ref{pd16}), (\ref{pd19}) and (\ref{pd116}) that
\be\la{pd117}\ba
N_3 & \le -\frac{\nu-\mu}{\nu} \frac{d}{dt} \left( \int G \p_i u \cdot \na u^i dx \right)
-\frac{\nu-\mu}{2\nu^2} \| \dot{G} \|^2_{L^2}
+ \frac{\mu}{4} \| \na \dot{u} \|_{L^2}^{2} \\
& \quad + C \| \sqrt{\n} \dot{u} \|^2_{L^2} \| \na u \|^2_{L^2} + \frac{C}{\nu} \| G \|^4_{L^4}
+ \frac{C}{\nu^3} \| P \|^4_{L^4} + C \| \na u \|^4_{L^4}.
\ea\ee
Then, substituting (\ref{pd14}), (\ref{pd15}) and (\ref{pd117}) into (\ref{pd13}) results in
\be\la{pd118}\ba
&\left(\int\rho|\dot{u}|^2dx + \frac{2(\nu-\mu)}{\nu} \int G \p_i u \cdot \na u^i dx \right)_t
+ \frac{3\mu}{4} \| \na \dot{u} \|_{L^2}^2 + \frac{\nu-\mu}{\nu^2} \| \dot{G} \|^2_{L^2} \\
& \le C \| \sqrt{\n} \dot{u} \|^2_{L^2} \| \na u \|^2_{L^2} + \frac{C}{\nu} \| G \|^4_{L^4}
+ \frac{C}{\nu^3} \| P \|^4_{L^4} + C \| \na u \|^4_{L^4} \\
& \le C \left( \| \na u \|^2_{L^2} + \frac{1}{\nu} \| G \|^2_{L^2}  \right) \| \sqrt{\n} \dot{u} \|^2_{L^2}
+ \frac{C}{\nu^3} \| P \|^4_{L^4},
\ea\ee
where in the last inequality we have used the following estimate
\be\nonumber\ba
\frac{C}{\nu} \| G \|^4_{L^4} + C \| \na u \|^4_{L^4}
& \le \frac{C}{\nu} \| G \|^4_{L^4} + C \| \div u \|^4_{L^4} + C \| \o \|^4_{L^4} \\
& \le \frac{C}{\nu} \| G \|^4_{L^4} + \frac{C}{\nu^4} \| P \|^4_{L^4}
+ C \| \o \|^2_{L^2} \| \na \o \|^2_{L^2} \\
& \le \frac{C}{\nu} \| G \|^2_{L^2} \| \sqrt{\n} \dot{u} \|^2_{L^2} + \frac{C}{\nu^4} \| P \|^4_{L^4}
+ C \| \o \|^2_{L^2} \| \sqrt{\n} \dot{u} \|^2_{L^2},
\ea\ee
due to (\ref{dc1}), (\ref{gn11}), (\ref{gw}), and (\ref{p2}).

Multiplying (\ref{pd118}) by $t^2$ and using (\ref{cp032}) and (\ref{gw}), we derive
\be\la{pd119}\ba
&\left( t^2 \int\rho|\dot{u}|^2dx + \frac{2(\nu-\mu)}{\nu} t^2 \int G \p_i u \cdot \na u^i dx \right)_t
+ \frac{3\mu}{4} t^2 \| \na \dot{u} \|_{L^2}^2 \\
& \le 2 t \left(\int\rho|\dot{u}|^2dx + \frac{2(\nu-\mu)}{\nu} \int G \p_i u \cdot \na u^i dx \right)
+ C t \| \sqrt{\n} \dot{u} \|^2_{L^2} + \frac{C}{\nu^3} t^2 \| P \|^4_{L^4}.
\ea\ee
In addition, we deduce from (\ref{gw}), (\ref{p2}) and Young's inequality that
\be\la{pd120}\ba
\frac{2(\nu-\mu)}{\nu} \int G \p_i u \cdot \na u^i dx 
& \le 2 \left| \int \left( G (\div u)^2 + 2G\na u^1 \cdot \na^\bot u^2 \right) dx \right| \\
& \le \frac{C}{\nu^2} \| G \|^3_{L^3} + \frac{C}{\nu^2} \| P \|^3_{L^3}
+ C \| \na G \|_{L^2} \| \na u \|^2_{L^2} \\
& \le \frac{C}{\nu^2} \| G \|^2_{L^2} \| \na G \|_{L^2} + \frac{C}{\nu^2} \| P \|^3_{L^3}
+C \| \sqrt{\n} \dot{u} \|_{L^2} \| \na u \|^2_{L^2} \\
& \le \frac14 \| \sqrt{\n} \dot{u} \|^2_{L^2}
+\frac{C}{\nu^4} \| G \|^4_{L^2} + \frac{C}{\nu^2} \| P \|^3_{L^3}
+ C \| \na u \|^4_{L^2}.
\ea\ee
Integrating (\ref{pd119}) over $(0,T)$ and using (\ref{cp032}), (\ref{pd120}), and H\"older's inequality, we derive (\ref{pd11}).

It remains to prove (\ref{pd011}).
From (\ref{dc1}), (\ref{gw}), and (\ref{gn11}), we deduce that for any $2 \le r<\infty$,
\be\la{pd121}\ba
\| \na u \|^r_{L^r}
& \le C \| \div u \|^r_{L^r} + C \| \o \|^r_{L^r} \\
& \le \frac{C}{\nu^r} \left( \| G \|^r_{L^r} + \| P \|^r_{L^r} \right) + C \| \o \|^2_{L^2} \| \na \o \|^{r-2}_{L^2} \\
& \le \frac{C}{\nu^r} \| G \|^2_{L^2} \| \na G \|^{r-2}_{L^2} + \frac{C}{\nu^r} \| P \|^r_{L^r} + C \| \na u \|^2_{L^2} \| \sqrt{\n} \dot{u} \|^{r-2}_{L^2} \\
& \le C \left( \frac{1}{\nu} \| G \|^2_{L^2} + \| \na u \|^2_{L^2} \right) \| \sqrt{\n} \dot{u} \|^{r-2}_{L^2} + \frac{C}{\nu^r} \| P \|^r_{L^r}.
\ea\ee
We claim that for $m \in \mathbb{N}^+$,
\be\la{pd122}\ba
\sup_{\si(T) \le t \le T} \left( \frac{1}{\nu^m} t^m \| P \|^{m+1}_{L^{m+1}} \right)
+ \int_{\si(T)}^T \frac{1}{\nu^{m+1}} t^{m} \| P \|^{m+2}_{L^{m+2}} dt \le C.
\ea\ee
Combining (\ref{pd121}), (\ref{pd122}), and H\"older's inequality gives (\ref{pd011}).

We shall prove (\ref{pd122}) by induction.
First, (\ref{sjgj01}) ensures that (\ref{pd122}) holds for $m=1$.
Assume that (\ref{pd122}) holds for $m=n$, that is,
\be\la{pd123}\ba
\sup_{\si(T) \le t \le T} \left( \frac{1}{\nu^n} t^n \| P \|^{n+1}_{L^{n+1}} \right)
+ \int_{\si(T)}^T \frac{1}{\nu^{n+1}} t^{n} \| P \|^{n+2}_{L^{n+2}} dt \le C.
\ea\ee
Choosing $p=n+2$ in (\ref{sjgj5}), multiplying the resulting inequality by $\frac{1}{\nu^n} t^{n+1}$, and using (\ref{cp032}), (\ref{sjgj01}), and (\ref{pd11}), we arrive at
\be\la{pd124}\ba
& \frac{d}{dt} \left( \frac{1}{\nu^{n+1}} t^{n+1} \| P \|^{n+2}_{L^{n+2}} \right)
+ \frac{(n+2)\ga-1}{2 \nu^{n+2}} t^{n+1} \| P \|^{n+3}_{L^{n+3}} \\
& \le \frac{n+1}{\nu^{n+1}} t^n \| P \|^{n+2}_{L^{n+2}}
+ \frac{C}{\nu^{n+2}} t^{n+1} \| G \|^{2}_{L^2} \| \sqrt{\n} \dot{u} \|^{n+1}_{L^2} \\
& \le \frac{n+1}{\nu^{n+1}} t^n \| P \|^{n+2}_{L^{n+2}} + C \left( \| \na u \|^{2}_{L^2} + \frac{1}{\nu} \| P \|^2_{L^2} \right).
\ea\ee
Integrating (\ref{pd124}) over $[\si(T),T]$ and using (\ref{pd123}), (\ref{cp1}), and (\ref{cp032}), we conclude that (\ref{pd122}) holds for $m=n+1$.
By induction, we obtain (\ref{pd011}) and complete the proof of Lemma \ref{g1}.
\end{proof}

The following lemma plays a crucial role in deriving the upper bound for the density.
\begin{lemma}\la{l4}
Let $(\n,u)$ be a classical solution of \eqref{ns} on $\rr \times (0,T]$
with $ \tilde{\n}=0 $.
Then for any $r \in [2,\infty)$, there exists a positive constant $C$ depending on
$a$, $\Vert {\bar{x}}^a \rho_0 \Vert_{L^1},\ N_0$, $E_0$, $\| \n_0 \|_{L^\infty}$, and $r$
such that for all $t \in (0,T]$,
\be\la{cp4}\ba
\left( \int_{\rr}\rho |v|^r dx\right) ^{1/r} 
\le C (1+t)^4 \left( \| \sqrt{\rho} v \| _{L^2} + \| \nabla v \|_{L^2}\right),
\ea\ee
holds for any $v\in \left. \left\{ v\in D^1 ({\rr})\right| \sqrt{\rho} v \in L^2(\rr) \right\}$.
\end{lemma}
\begin{proof}
First, for any integer $N>1$, let $\varphi_N$ be a smooth function  satisfies:
\be\la{cp41}\ba
0 \le \varphi_N \le 1, \quad \varphi_N=
\begin{cases}
1,\quad &\mathrm{ if \  } |x| \le N,\\
0,\quad &\mathrm{ if \  } |x| > 2N,
\end{cases}
\quad |\na \varphi_N | \le 2 N^{-1}.
\ea\ee
Multiplying $(\ref{ns})_1$ by $\varphi_N$ and integrating over $\rr$,
we deduce from (\ref{cp1}) and (\ref{cp41}) that
\be\la{cp42}\ba
\frac{d}{dt} \int \n \varphi_N dx =\int \n u \cdot \na \varphi_N dx 
\ge -2 N^{-1} \left( \int \n dx \right)^{\frac{1}{2}} 
\left( \int \n |u|^2 dx \right)^{\frac{1}{2}} \ge -2 \hat{C} N^{-1},
\ea\ee
which implies that for any $0 \le t \le T$,
\be\la{cp43}\ba
\int \n \varphi_N dx \ge \int \n_0 \varphi_N dx -2 \hat{C} N^{-1}t,
\ea\ee
where the positive constant $\hat{C}$ depends only on $\| {\bar{x}}^a \rho_0 \|_{L^1}$ and $E_0$.

Then, for $\tilde{N} \triangleq 4(1+N_0+4\hat{C}t)$, where $N_0$ is defined in \eqref{rho00},
we use (\ref{rho00}) to derive
\be\la{cp44}\ba
\int_{B_{N_1}} \n dx \ge \int \n \varphi_{\tilde{N}/2} dx
& \ge \int \n_0 \varphi_{\tilde{N}/2} dx - 4 \hat{C} \tilde{N}^{-1}t \\
& \ge \int \n_0 \varphi_{N_0} dx - 4 \hat{C} \tilde{N}^{-1}t \\
& \ge \int_{B_{N_0}} \n_0 dx - 4 \hat{C} \tilde{N}^{-1}t \ge \frac{1}{4}.
\ea\ee
Consequently, there exists a positive constant $N_1$ depending on
$N_0$, $\| {\bar{x}}^a \rho_0 \|_{L^1}$, and $E_0$,
such that for all $t \in (0,T]$,
\be\la{cp45}\ba
\int _{B_{N_1(1+t)}}\rho (x,t) dx \ge \frac{1}{4}.
\ea\ee
From (\ref{cp45}), (\ref{cp1}), and Lemma \ref{esnr}, we conclude that for any $r \in [2,\infty)$ and $\beta>0$,
\be\la{cp46}\ba
\left( \int _{\rr}\rho |v|^r dx\right) ^{1/r} 
& \le C (1+t)^3 (1+\Vert {\bar{x}}^\beta \rho \Vert_{L^1}) 
\left( \Vert \sqrt{\rho} v \Vert _{L^2} 
+ \Vert \nabla v \Vert _{L^2}\right).
\ea\ee
Next, multiplying $(\ref{ns})_1$ by $\left(1+|x|^2 \right)^{\frac{1}{2}}$
and integrating the resulting equation over $\rr$, we arrive at
\be\la{cp47}\ba
\frac{d}{dt}\int \rho (1+|x|^2)^{\frac{1}{2}} dx
&\le \int |x| (1+|x|^2)^{-\frac{1}{2}} \rho |u| dx \\
&\le \left( \int \rho dx\right) ^{1/2}\left( \int \rho |u|^2 dx\right) ^{1/2} \\
&=\left( \int \rho_0 dx\right) ^{1/2}\left( \int \rho |u|^2 dx\right) ^{1/2} \\
&\le \left( \int {\bar{x}}^a \rho_0 dx\right) ^{1/2}\left( \int \rho |u|^2 dx\right) ^{1/2} \\
&\le C,
\ea\ee
where in the last inequality we have used (\ref{cp1}).

Integrating (\ref{cp47}) over $(0,t)$ shows
\be\la{cp48}\ba
\int \rho (1+|x|^2)^{\frac{1}{2}} dx 
& \le \int \rho_0 (1+|x|^2)^{\frac{1}{2}} dx + C t \\
& \le \int {\bar{x}}^a \rho_0 dx + C t \\
& \le C(1+t).
\ea\ee
Combining this with (\ref{cp46}) and choosing $\beta$ sufficiently small, we obtain (\ref{cp4}) and finish the proof of Lemma \ref{l4}.
\end{proof}

With Lemmas \ref{l1}--\ref{l4} at hand, we are in a position to prove Proposition \ref{l5}.

\begin{proof}[Proof of Proposition \ref{l5}]
First, we use (\ref{gw}) to rewrite $(\ref{ns})_1$ as
\be\ba\la{2cp61}
\pa_t \log\n + u\cdot\na \log\n + \frac{1}{\nu}(\n^\ga + G) = 0.
\ea\ee
By making use of the fact that $\n^\ga \geq \ga\log\n + 1$, we have
\be\ba\la{2cp62}
\frac{d}{ds} \log \n(s) + \frac{\ga}{\nu}\log \n
\le \frac{1}{\nu} \| G \|_{L^\infty}.
\ea\ee
Applying the maximum principle to (\ref{2cp62}) leads to
\be\ba\la{2cp63}
\log \n(t)
& \le e^{-\frac{\ga}{\nu}t} \log \n(0)
+ \frac{1}{\nu} \int_0^t e^{-\frac{\ga}{\nu}(t-s)} \| G \|_{L^\infty} ds.
\ea\ee
Moreover, from (\ref{cp051}), (\ref{cp4}), and (\ref{p2}), we conclude that for any $2 \le p < \infty$,
\be\la{2cp64}\ba
\| \na G \|_{L^p} \le C \| \n \dot{u} \|_{L^p}
\le C(1+t)^4 \left( \| \sqrt{\n} \dot{u} \|_{L^2} + \| \na \dot{u} \|_{L^2} \right).
\ea\ee
Combining (\ref{gw}), (\ref{gn11}), (\ref{cp031}), (\ref{pd11}), (\ref{2cp64}), and the Gagliardo-Nirenberg inequality gives
\be\ba\la{2cp65}
\int_0^{\si(t)} \| G \|_{L^\infty} ds
& \le C \int_0^{\si(t)} \| G \|_{L^2}^{3/8} \| \na G \|_{L^5}^{5/8} ds \\
& = C \int_0^{\si(t)} \| G \|_{L^2} ^{3/8} \left(\si^2\| \na G \|^{2}_{L^5} \right)^{5/16} \si^{ -5/8 } ds \\
& \le C \nu ^{3/8} \int_0^{\si(t)} \left( \si^2 \| \sqrt{\n} \dot u \|^{2}_{L^2} + \si^2 \| \na \dot u \|^{2}_{L^2} \right)^{5/16}
\si^{-5/8} ds \\
& \le C \nu^{3/8} \left( \int_0^1 \si^{-10/11} ds \right)^{11/16} \\
& \le C \nu^{3/8},
\ea\ee
and
\be\la{2cp66}\ba
\int_{\si(T)}^T \| G \|^4_{L^\infty} dt
& \le C \int_{\si(T)}^T \| G \|^{35/9}_{L^{72}} \| \na G \|^{1/9}_{L^{72}} dt \\
& \le C \int_{\si(T)}^T \| G \|^{35/9}_{L^{72}}
t^{4/9} \left( \| \sqrt{\n} \dot{u} \|_{L^2} + \| \na \dot{u} \|_{L^2} \right)^{1/9} dt \\
& \le C \nu^{1/18} \int_{\si(T)}^T t^{-27/18}
\left( \| \sqrt{\n} \dot{u} \|_{L^2} + \| \na \dot{u} \|_{L^2} \right)^{1/9} dt \\
& \le C \nu^{1/18},
\ea\ee
where in the third inequality we have used the following estimate:
\be\nonumber\ba
\| G \|_{L^{72}} \le C \| G \|^{1/36}_{L^2} \| \na G \|^{1-1/36}_{L^2}
\le C \| G \|^{1/36}_{L^2} \| \sqrt{\n} \dot{u} \|^{1-1/36}_{L^2}
\le C \nu^{1/72} t^{-1+1/72},
\ea\ee
due to (\ref{gn11}), (\ref{p2}), (\ref{sjgj01}), and (\ref{pd11}).

By virtue of (\ref{2cp65}), (\ref{2cp66}), and H\"older's inequality, it holds that
\be\la{2cp67}\ba
\int_0^t e^{-\frac{\ga}{\nu}(t-s)} \| G \|_{L^\infty} ds
& = \int_0^{\si(t)} e^{-\frac{\ga}{\nu}(t-s)} \| G \|_{L^\infty} ds
+ \int_{\si(t)}^t e^{-\frac{\ga}{\nu}(t-s)} \| G \|_{L^\infty} ds \\
& \le C \int_0^{\si(t)} \| G \|_{L^\infty} ds
+ \left( \int_{\si(t)}^t e^{-\frac{4\ga}{3\nu}(t-s)} ds \right)^{\frac{3}{4}}
\left( \int_{\si(t)}^t \| G \|^4_{L^\infty} ds \right)^{\frac{1}{4}} \\
& \le C \nu^{\frac{5}{6}}.
\ea\ee
Substituting (\ref{2cp67}) into (\ref{2cp63}) yields
\be\ba\la{2cp68}
\log \n(t)
& \le e^{-\frac{\ga}{\nu}t} \log \left( 1 + \| \n_0 \|_{L^\infty} \right)
+ C \nu^{-\frac{1}{6}} \\
& \le \log \left( 1 + \| \n_0 \|_{L^\infty} \right)
+ M_1 \nu^{-\frac{1}{6}},
\ea\ee
where $M_1$ is a positive constant depending only on
$\mu$, $\ga$, $a$, $\Vert {\bar{x}}^a \rho_0 \Vert_{L^1}$, $N_0$, $E_0$, $\| \n_0 \|_{L^\infty}$, and $\|\na u_0\|_{L^2}$,
but is independent of $T$ and $\nu$.

Finally, we define
\be\ba\la{2cp69}
\nu_1 \triangleq \left( \frac{M_1}{\log \frac{3}{2}} \right)^6,
\ea\ee
which ensures (\ref{cp052}) when $\nu \geq \nu_1$.
This completes the proof of Proposition \ref{l5}.
\end{proof}

\subsection{Far-field density is away from vacuum}

In this subsection, we assume that $\tilde{\n}>0$, and that the initial data $(\n_0,u_0)$ satisfy (\ref{wsol1}) and $(\ref{wsol01})_2$ with $\n_0>0$.

We define
\be\ba\nonumber
E(T) \triangleq \sup_{0 \le t \le T} \si \left( \| \na u \|^2_{L^2} + \nu \| \div u \|^2_{L^2} \right) + \int_{0}^T \si \int \n |\dot{u}|^2 dx dt.
\ea\ee

We derive the following key a priori estimates on $(\n,u)$, which guarantees the local solution $(\n,u)$ to the global one.
\begin{proposition}\label{1cpp1}
There are two generic positive constants $\nu_2$ and $\mathbf{C_1}$ depending only on $\ga$, $\mu$, $E_0$, $\| \n_0 \|_{L^\infty}$, $\| \na u_0 \|_{L^2}$, and $\tilde{\n}$ such that if $(\rho,u)$ is a classical solution of \eqref{ns}--\eqref{i3} on $\rr \times(0,T]$ satisfying
\be\la{1cpp2}\ba
\sup_{\rr \times [0,T]} \n \leq 2 \left( 1 + \| \n_0 \|_{L^\infty} \right) e^{\frac{1}{\ga} \tilde{\n}^\ga}, \quad
E(T) \le 2 \mathbf{C_1},
\ea\ee
the following estimate holds:
\be\la{1cpp3}\ba
\sup_{\rr \times [0,T]} \n \leq \frac{3}{2} \left( 1 + \| \n_0 \|_{L^\infty} \right) e^{\frac{1}{\ga} \tilde{\n}^\ga}, \quad
E(T) \le \mathbf{C_1},
\ea\ee
provided that $\nu \ge \nu_2$.
\end{proposition}
\begin{proof}
Proposition \ref{1cpp1} is an easy consequence of the following Lemmas \ref{cpl3} and \ref{cpl5}, with $\nu_2$ as in (\ref{1cp69}) and $\mathbf{C_1}$ as in (\ref{1cp324}).
\end{proof}

We begin with the following standard energy estimate.
\begin{lemma}\la{cpl1}
Suppose that $(\n,u)$ is a classical solution to \eqref{ns} 
on  $\rr \times (0,T]$, then the following holds:
\be\ba\la{1cp1}
\sup_{0\leq t\leq T} \int \left( \frac{1}{2}\rho |u|^2 + H(\n) \right) dx
+ \int_0^T  \left( \mu \|\na u\|^2_{L^2} + (\mu+\lam) \|\div u\|^2_{L^2} \right) dt \leq E_0,
\ea\ee
where $E_0$ is defined by \eqref{e0}.
\end{lemma}
\begin{proof}
Multiplying $(\ref{ns})_2$ by $u$ and integrating by parts over $\rr$, we obtain (\ref{1cp1}) after using $(\ref{ns})_1$.
\end{proof}

\begin{lemma}\la{cpl2}
Let $(\n,u)$ be a classical solution to \eqref{ns}--\eqref{i3} satisfying \eqref{1cpp2}.
Then there exists a positive constant $\tilde{C}$ depending only on
$\ga$, $\mu$, $E_0$, $\| \n_0 \|_{L^\infty}$, and $\tilde{\n}$ such that
\be\la{1cp02}\ba
\sup_{0 \le t \le \si(T)} \left( \| \na u \|^2_{L^2} + \nu \| \div u \|^2_{L^2} \right)
+ \int_0^{\si(T)} \| \sqrt{\n} \dot{u}\|^2_{L^2} dt \le \tilde{C} \left( 1 + \nu \| \div u_0 \|^2_{L^2} + \| \na^\bot \cdot u_0 \|^2_{L^2} \right),
\ea\ee
and
\be\la{1cp002}\ba
\sup_{0 \le t \le \si(T)} t \left( \| \na u \|^2_{L^2} + \nu \| \div u \|^2_{L^2} \right)
+ \int_0^{\si(T)} t \| \sqrt{\n} \dot{u}\|^2_{L^2} dt \le \tilde{C}.
\ea\ee
\end{lemma}
\begin{proof}
Following the proof of Lemma \ref{l3}, we can obtain
\be\la{1cp21}\ba
\frac{d}{dt} A^2_1 + \| \sqrt{\n} \dot{u}\|^2_{L^2}
\le C A^2_1 \| \na u\|^2_{L^2}+C\| \na u\|^2_{L^2}+C\| P-P(\tilde{\n}) \|^2_{L^2},
\ea\ee
where $A_1$ is given by (\ref{defa1}).

By virtue of (\ref{1cpp2}), (\ref{1cp1}), (\ref{qkjsn}), and (\ref{gw}), it holds that
\be\la{1cp27}\ba
\| P-P( \tilde{\n}) \|^2_{L^2} \le C \| \n - \tilde{\n} \|^2_{L^2} \le C,
\ea\ee
and
\be\ba\la{1cp29}
\| \na u\|^2_{L^2} & \le C \left( \| \div u\|^2_{L^2}+\| \o \|^2_{L^2} \right) \\
& \le \frac{C}{\nu^2} \left( \| G \|^2_{L^2}+\| P-P( \tilde{\n}) \|^2_{L^2} \right) + C \| \o \|^2_{L^2} \\
& \le C A^2_1 + C.
\ea\ee
Combining (\ref{1cp21}) with (\ref{1cp27}) gives
\be\la{1cp210}\ba
& \frac{d}{dt} A^2_1 + \| \sqrt{\n} \dot{u}\|^2_{L^2}
\le C A^2_1 \| \na u \|^2_{L^2} + C \| \na u \|^2_{L^2} + C.
\ea\ee
Applying Gr\"onwall's inequality to (\ref{1cp210}) over $(0,\si(T))$ together with (\ref{1cp1}) and (\ref{1cp29}) gives (\ref{1cp02}).

On the other hand, multiplying (\ref{1cp210}) by $t$ shows
\be\la{1cp211}\ba
\frac{d}{dt} \left( t A^2_1 \right) + t \| \sqrt{\n} \dot{u}\|^2_{L^2}
\le C t A^2_1 \| \na u \|^2_{L^2} + C t \left( \| \na u \|^2_{L^2} + 1 \right) + \nu \| \div u \|^2_{L^2},
\ea\ee
Applying Gr\"onwall's inequality to (\ref{1cp211}) over $(0,\si(T))$ and using (\ref{1cp1}) and (\ref{1cp29}), we obtain (\ref{1cp002}).
This completes the proof of Lemma \ref{cpl2}.
\end{proof}

\begin{lemma}\la{cpl3}
There exist two positive constants $\hat{\nu}_2$ and $\mathbf{C_1}$ depending only on
$\ga$, $\mu$, $E_0$, $\| \n_0 \|_{L^\infty}$, and $\tilde{\n}$ such that,
if $(\n,u)$ is a classical solution to \eqref{ns}--\eqref{i3} satisfying \eqref{1cpp2}, then
\be\la{1cp03}\ba
E(T) \le \mathbf{C_1},
\ea\ee
provided $\nu \ge \hat{\nu}_2$.
\end{lemma}
\begin{proof}
First, if $T \le 1$, Lemma \ref{cpl2} directly gives (\ref{1cp03}) with $\mathbf{C_1} = \tilde{C}$.

Next, assume that $T>1$.
From (\ref{1cp002}), we obtain
\be\la{1cp0310}\ba
\sup_{0 \le t \le 1} t \left( \| \na u \|^2_{L^2} + \nu \| \div u \|^2_{L^2} \right)
+ \int_0^1 t \| \sqrt{\n} \dot{u}\|^2_{L^2} dt \le C.
\ea\ee
Multiplying $(\ref{ns})_2$ by $\dot{u}$ yields
\be\la{1cp31}\ba
\int \n |\dot{u}|^2 dx = - \int \dot{u} \cdot \na (P-P(\tilde{\n})) dx + \mu \int \Delta u \cdot \dot{u} dx + (\mu+\lam) \int \na \div u \cdot \dot{u} dx.
\ea\ee
Note that $P-P(\tilde{\n})$ satisfies
\be\la{1cp32}\ba
(P-P(\tilde{\n}))_t + u \cdot \na (P-P(\tilde{\n})) + \ga (P-P(\tilde{\n})) \div u + \ga P(\tilde{\n}) \div u = 0.
\ea\ee
Integrating by parts and using (\ref{1cp32}) lead to
\be\la{1cp33}\ba
- \int \dot{u} \cdot \na (P-P(\tilde{\n})) dx
& = \int \left( (\div u)_t (P-P(\tilde{\n})) - (u \cdot \na u) \cdot \na (P-P(\tilde{\n})) \right) dx \\
& = \left( \int \div u (P-P(\tilde{\n})) dx \right)_t
+ \int u \cdot \na (P-P(\tilde{\n})) \div u dx \\
& \quad + \int \ga P (\div u)^2 dx
- \int (u \cdot \na u) \cdot \na (P-P(\tilde{\n})) dx \\
& = \left( \int \div u (P-P(\tilde{\n})) dx \right)_t
- \int \left( (P-P(\tilde{\n})) - \ga P \right) (\div u)^2 dx \\
& \quad + \int \p_i u^j \p_j u^i (P-P(\tilde{\n})) dx \\
& \le \left( \int \div u (P-P(\tilde{\n})) dx \right)_t + C \| \na u \|^2_{L^2}.
\ea\ee
Integration by parts also implies
\be\la{1cp34}\ba
\mu \int \Delta u \cdot \dot{u} dx
& = -\frac{\mu}{2} \left( \| \na u \|^2_{L^2} \right)_t
- \mu \int \p_i u^j \p_i(u^k \p_k u^j) dx \\
& = -\frac{\mu}{2} \left( \| \na u \|^2_{L^2} \right)_t
- \mu \int \left( \p_i u^j \p_i u^k \p_k u^j + \p_i u^j u^k \p_k \p_i u^j \right) dx \\
& = -\frac{\mu}{2} \left( \| \na u \|^2_{L^2} \right)_t
- \mu \int \p_i u^j \p_i u^k \p_k u^j dx
+ \frac{\mu}{2} \int \div u |\na u|^2 dx \\
& = -\frac{\mu}{2} \left( \| \na u \|^2_{L^2} \right)_t
- \frac{\mu}{2} \int \div u |\na u|^2 dx
- \mu \int \div u \na u^1 \cdot \na^\bot u^2 dx,
\ea\ee
where in the last equality we have used the following fact:
\be\la{1cp35}\ba
\p_i u^j \p_i u^k \p_k u^j = \div u |\na u|^2 + \div u \na u^1 \cdot \na^\bot u^2.
\ea\ee
Moreover, we also have
\be\la{1cp36}\ba
& (\mu+\lam) \int \na \div u \cdot \dot{u} dx \\
& = - \frac{(\mu+\lam)}{2} \left( \| \div u \|^2_{L^2} \right)_t
- (\mu+\lam) \int \div u \div (u \cdot \na u) dx \\
& = -\frac{(\mu+\lam)}{2} \left( \| \div u \|^2_{L^2} \right)_t
- (\mu+\lam) \int \left( \div u \p_i u^j \p_j u^i + \div u u \cdot \na \div u \right) dx \\
& = -\frac{(\mu+\lam)}{2} \left( \| \div u \|^2_{L^2} \right)_t
- (\mu+\lam) \int \div u \p_i u^j \p_j u^i dx
+ \frac{(\mu+\lam)}{2} \int (\div u)^3 dx \\
& = -\frac{(\mu+\lam)}{2} \left( \| \div u \|^2_{L^2} \right)_t
- \frac{(\mu+\lam)}{2} \int (\div u)^3 dx
- 2 (\mu+\lam) \int \div u \na u^1 \cdot \na^\bot u^2 dx,
\ea\ee
where in the last equality we have used the identity
\be\la{1cp37}\ba
\p_i u^j \p_j u^i = (\div u)^2 + 2 \na u^1 \cdot \na^\bot u^2.
\ea\ee
Substituting (\ref{1cp33}), (\ref{1cp34}), and (\ref{1cp36}) into (\ref{1cp31}) yields
\be\la{1cp38}\ba
& \frac{d}{dt} \left( \frac{\mu}{2} \| \na u \|^2_{L^2} + \frac{(\mu+\lam)}{2} \| \div u \|^2_{L^2} - \int \div u (P-P(\tilde{\n})) dx \right)
+ \int \n |\dot{u}|^2 dx \\
& \le C \| \na u \|^2_{L^2} - \frac{\mu}{2} \int \div u |\na u|^2 dx
- \frac{(\mu+\lam)}{2} \int (\div u)^3 dx
- (3\mu+2\lam) \int \div u \na u^1 \cdot \na^\bot u^2 dx.
\ea\ee
Since $(\ref{gw})$ implies
\be\la{1cp380}\ba
\div u = \frac{1}{\nu} (G + (P-P(\tilde{\n}))).
\ea\ee
By virtue of (\ref{gw}), (\ref{1cp380}), and H\"older's inequality, it holds that
\be\la{1cp39}\ba
- \frac{\mu}{2} \int \div u |\na u|^2 dx
& = - \frac{\mu}{2 \nu} \int (G + (P-P(\tilde{\n}))) |\na u|^2 dx \\
& \le C \| \na u \|^2_{L^2} + \frac{C}{\nu} \| G \|_{L^4} \| \na u \|_{L^2} \| \na u \|_{L^4} \\
& \le C \| \na u \|^2_{L^2} + \frac{C}{\nu^3} \| G \|^4_{L^4} + \frac{C}{\nu} \| \na u \|^4_{L^4} \\
& \le C \left( 1 + \| \na u \|^2_{L^2} \right) \| \sqrt{\n} \dot{u} \|^2_{L^2} + \frac{C}{\nu^2} \| P-P(\tilde{\n}) \|^4_{L^4} + C \| \na u \|^2_{L^2},
\ea\ee
where in the third inequality we have used the following estimates:
\be\la{1cp310}\ba
\| G \|^4_{L^4} \le C \| G \|^2_{L^2}\| \na G \|^2_{L^2}
& \le C \nu^2 \left( \| \div u \|^2_{L^2} + \| P-P(\tilde{\n}) \|^2_{L^2} \right) \| \sqrt{\n} \dot{u} \|^2_{L^2} \\
& \le C \nu^2 \left( 1 + \| \div u \|^2_{L^2} \right) \| \sqrt{\n} \dot{u} \|^2_{L^2},
\ea\ee
and
\be\la{1cp311}\ba
\|\na u\|^4_{L^4} & \le C \left( \|\div u\|^4_{L^4} + \|\o \|^4_{L^4} \right) \\
& \le \frac{C}{\nu^4} \left( \| G \|^4_{L^4} + \| P-P(\tilde{\n}) \|^4_{L^4} \right)
+ C \|\o \|^2_{L^2} \| \na \o \|^2_{L^2} \\
& \le C \left( 1 + \| \div u \|^2_{L^2} + \| \o \|^2_{L^2} \right) \| \sqrt{\n} \dot{u} \|^2_{L^2} + \frac{C}{\nu^4} \| P-P(\tilde{\n}) \|^4_{L^4},
\ea\ee
due to (\ref{gn11}), (\ref{p2}), and (\ref{1cp27}).

It follows from (\ref{1cp380}), (\ref{1cp310}), and Cauchy's inequality that
\be\la{1cp312}\ba
- \frac{(\mu+\lam)}{2} \int (\div u)^3 dx
& \le \frac{C}{\nu} \int \left( |G|^2 + | P-P(\tilde{\n}) |^2 \right) |\div u| dx \\
& \le \frac{C}{\nu^3} \left( \| G \|^4_{L^4} + \| P-P(\tilde{\n}) \|^4_{L^4} \right) + C \nu \| \div u \|^2_{L^2} \\
& \le \frac{C}{\nu} \left( 1 + \| \div u \|^2_{L^2} \right) \| \sqrt{\n} \dot{u} \|^2_{L^2} + \frac{C}{\nu^3} \| P-P(\tilde{\n}) \|^4_{L^4} + C \nu \| \div u \|^2_{L^2}.
\ea\ee
Furthermore, (\ref{1cp380}) and Young's inequality give
\be\la{1cp313}\ba
- (3\mu+2\lam) \int \div u \na u^1 \cdot \na^\bot u^2 dx
& = - \frac{3\mu+2\lam}{\nu} \int (G + (P-P(\tilde{\n}))) \na u^1 \cdot \na^\bot u^2 dx \\
& \le C \| \na u \|^2_{L^2} + C \| G \|_{\mathcal{BMO}} \|\na u^1\cdot \na^{\bot}u^2\|_{\mathcal{H}^1} \\
& \le C \| \na u \|^2_{L^2} + C \| \sqrt{\n} \dot{u} \|_{L^2} \| \na u \|^2_{L^2} \\
& \le \frac{1}{2} \| \sqrt{\n} \dot{u} \|^2_{L^2} + C \| \na u \|^4_{L^2} + C \| \na u \|^2_{L^2}.
\ea\ee
Putting (\ref{1cp39}), (\ref{1cp312}), and (\ref{1cp313}) into (\ref{1cp38}) shows
\be\la{1cp314}\ba
& \frac{d}{dt} \left( \frac{\mu}{2} \| \na u \|^2_{L^2} + \frac{(\mu+\lam)}{2} \| \div u \|^2_{L^2} - \int \div u (P-P(\tilde{\n})) dx \right)
+ \frac{1}{2} \int \n |\dot{u}|^2 dx \\
& \le \frac{C}{\nu} \left( 1 + \| \na u \|^2_{L^2} \right) \| \sqrt{\n} \dot{u} \|^2_{L^2} + \frac{C}{\nu^3} \| P-P(\tilde{\n}) \|^4_{L^4} + C \| \na u \|^4_{L^2} \\
& \quad + C \left( \| \na u \|^2_{L^2} + \nu \| \div u \|^2_{L^2} \right).
\ea\ee
On the other hand, multiplying (\ref{1cp32}) by $3 (P-P(\tilde{\n}))^2$ and using (\ref{1cp380}), we derive
\be\la{1cp315}\ba
\frac{3\ga-1}{\nu} \| P-P(\tilde{\n}) \|^4_{L^4}
& = -\left( \int (P-P(\tilde{\n}))^3 dx \right)_t
- \frac{3\ga-1}{\nu} \int (P-P(\tilde{\n}))^3 G dx \\
& \quad - 3\ga P(\tilde{\n}) \int (P-P(\tilde{\n}))^2 \div u dx \\
& \le -\left( \int (P-P(\tilde{\n}))^3 dx \right)_t
+ \frac{\ga}{\nu} \| P-P(\tilde{\n}) \|^4_{L^4} + \frac{C}{\nu} \| G \|^4_{L^4} + C \nu \| \div u \|^2_{L^2},
\ea\ee
which yields
\be\la{1cp316}\ba
\frac{1}{\nu} \| P-P(\tilde{\n}) \|^4_{L^4}
& \le -\left( \int (P-P(\tilde{\n}))^3 dx \right)_t
+ \frac{C}{\nu} \| G \|^4_{L^4} + C \nu \| \div u \|^2_{L^2}.
\ea\ee
Combining (\ref{1cp314}) and (\ref{1cp316}) leads to
\be\la{1cp317}\ba
\frac{d}{dt} B(t) + \frac{1}{2} \int \n |\dot{u}|^2 dx
& \le \frac{C}{\nu} \left( 1 + \| \na u \|^2_{L^2} \right) \| \sqrt{\n} \dot{u} \|^2_{L^2} + C \| \na u \|^4_{L^2} + C \left( \| \na u \|^2_{L^2} + \nu \| \div u \|^2_{L^2} \right),
\ea\ee
with
\be\la{1cp318}\ba
B(t) \triangleq \frac{\mu}{2} \| \na u \|^2_{L^2} + \frac{(\mu+\lam)}{2} \| \div u \|^2_{L^2} - \int \div u (P-P(\tilde{\n})) dx
+ \frac{C}{\nu^2} \int (P-P(\tilde{\n}))^3 dx.
\ea\ee
Then, for any $t \in (1,T)$, integrating (\ref{1cp317}) over $(1,t)$ and using (\ref{1cp1}), we arrive at
\be\la{1cp319}\ba
& \frac{\mu}{2} \| \na u(t) \|^2_{L^2} + \frac{(\mu+\lam)}{2} \| \div u(t) \|^2_{L^2} - \int (\div u (P-P(\tilde{\n}))) (t) dx
+ \frac{1}{2} \int_{1}^t \| \sqrt{\n} \dot{u} \|^2_{L^2} ds \\
& \le B(1) + \frac{C}{\nu} \left( 1 + E_1(T) \right) E_1(T)
+ C \int_{1}^t \| \na u(\cdot,s) \|^4_{L^2} ds + C.
\ea\ee
Moreover, from (\ref{1cp27}), (\ref{1cp0310}), and Holder's inequality, we have
\be\la{1cp320}\ba
B(1) & = \frac{\mu}{2} \| \na u(\cdot,1) \|^2_{L^2} + \frac{(\mu+\lam)}{2} \| \div u(\cdot,1) \|^2_{L^2} \\
& \quad - \int (\div u (P-P(\tilde{\n})))(x,1) dx
+ \frac{C}{\nu^2} \int (P-P(\tilde{\n}))^3 (x,1) dx \\
& \le C \left( \| \na u(\cdot,1) \|^2_{L^2} + \nu \| \div u(\cdot,1) \|^2_{L^2} \right) + C \| (P-P(\tilde{\n}))(\cdot,1) \|^2_{L^2}
\le C,
\ea\ee
and
\be\la{1cp321}\ba
\int (\div u (P-P(\tilde{\n}))) (t) dx
\le \| \div u(\cdot,t) \|^2_{L^2} + C \| (P-P(\tilde{\n}))(\cdot,t) \|^2_{L^2}
\le \frac{1}{\nu} E_1(T) + C.
\ea\ee
Combining (\ref{1cp318}), (\ref{1cp319}), and (\ref{1cp320}) yields
\be\la{1cp322}\ba
& \| \na u(t) \|^2_{L^2} + \nu \| \div u(t) \|^2_{L^2} + \int_1^t \| \sqrt{\n} \dot{u} \|^2_{L^2} ds \\
& \le C_1 + \frac{C_2}{\nu} \left( 1 + E_1(T) \right) E_1(T)
+ C_3 \int_1^t \| \na u(\cdot,s) \|^4_{L^2} ds.
\ea\ee
Applying Gr\"onwall's inequality to (\ref{1cp322}), we obtain for any $t \in (1,T)$,
\be\la{1cp323}\ba
& \| \na u(t) \|^2_{L^2} + \nu \| \div u(t) \|^2_{L^2} + \int_1^t \| \sqrt{\n} \dot{u} \|^2_{L^2} ds \\
& \le \left( C_1 + \frac{C_2}{\nu} \left( 1 + E_1(T) \right) E_1(T) \right)
\exp{ \left( C_3 \int_1^t \| \na u(\cdot,s) \|^2_{L^2} ds \right) } \\
& \le \left( C_1 + \frac{C_2}{\nu} \left( 1 + E_1(T) \right) E_1(T) \right)
\exp{ \left( C_3 E_0 \right) }.
\ea\ee
We set
\be\la{1cp324}\ba
\mathbf{C_1} \triangleq C_1 \exp{ \left( C_3 E_0 \right) } + \tilde{C},
\ea\ee
and
\be\la{1cp325}\ba
\hat{\nu}_2 \triangleq 4 C_2 \exp{ \left( C_3 E_0 \right) } \left( 1 + 2 C_1 \exp{ \left( C_3 E_0 \right) } + 2 \tilde{C} \right).
\ea\ee
where $\tilde{C}$ is given in Lemma \ref{cpl2}.

Thus, it follows from (\ref{1cp002}), (\ref{1cp323}), (\ref{1cp324}), and (\ref{1cp325}) that (\ref{1cp03}) holds provided $\nu \ge \hat{\nu}_2$.
This completes the proof of Lemma \ref{cpl3}.
\end{proof}

\begin{lemma}\la{cpl4}
Let $(\n,u)$ be a classical solution to \eqref{ns}--\eqref{i3} satisfying \eqref{1cpp2}.
Then there exists a positive constant $C$ depending only on
$\ga$, $\mu$, $E_0$, $\| \n_0 \|_{L^\infty}$, and $\tilde{\n}$ such that
\be\ba\la{1cp04}
\sup_{0\le t\le T}
\si^2 \int\n|\dot u|^2dx
+\int_0^{T} \si^2 \| \na\dot u\|^2_{L^2} dt \le C.
\ea\ee
\end{lemma}
\begin{proof}
First, using (\ref{gw}), we rewrite $(\ref{ns})_2$ as
\be\la{1cp41}\ba
\n \dot{u} = \mu \Delta u+\frac{\nu-\mu}{\nu} \na G -\frac{\mu}{\nu} \na(P-P(\tilde{\n})).
\ea\ee

Following the proof for the periodic case, we apply the operator 
$ \dot u^j[\pa/\pa t+\div(u\cdot)]$
to $ (\ref{1cp41})^j,$ sum over $j$, and integrate by parts over $\rr$ to obtain
\be\la{1cp42}\ba
& \left(\frac{1}{2}\int\rho|\dot{u}|^2dx \right)_t\\
& = \mu\int\dot{u}^j \left[\Delta u_t^j + \text{div}(u\Delta u^j) \right] dx 
-\frac{\mu}{\nu}\int\dot{u}^j \left[\p_jP_t+\text{div}(u \p_j (P-P(\tilde{\n})) \right]dx \\
& \quad + \frac{\nu-\mu}{\nu} \int \dot{u}^j \left[ \p_t\p_j G + \text{div} (u\p_j G) \right] dx \\
& \triangleq\sum_{i=1}^{3}N_i.
\ea\ee
By virtue of (\ref{1cpp2}), (\ref{p2}), and (\ref{1cp380}), and adapting the argument of Lemma \ref{g1}, we arrive at
\be\la{1cp43}\ba
N_1 + N_2 \le -\frac{5 \mu}{8} \| \na \dot{u} \|_{L^2}^2 + C \| \na u \|_{L^4}^4 + \frac{C}{\nu^4} \| P-P(\tilde{\n}) \|^4_{L^4}.
\ea\ee
and
\be\la{1cp44}\ba
N_3 & \le -\frac{\nu-\mu}{\nu} \frac{d}{dt} \left( \int G \p_i u \cdot \na u^i dx \right)
- \frac{\nu-\mu}{2\nu^2} \| \dot{G} \|^2_{L^2}
+ \frac{\mu}{4} \| \na \dot{u} \|_{L^2}^{2} \\
& \quad + C \| \sqrt{\n} \dot{u} \|^2_{L^2} \| \na u \|^2_{L^2} + \frac{C}{\nu} \| G \|^4_{L^4}
+ \frac{C}{\nu^3} \| P-P(\tilde{\n}) \|^4_{L^4} + C \| \na u \|^4_{L^4}.
\ea\ee
Substituting (\ref{1cp43}) and (\ref{1cp44}) into (\ref{1cp42}) and using (\ref{1cp310}), (\ref{1cp311}), and (\ref{1cp316}) imply
\be\la{1cp45}\ba
&\left(\int\rho|\dot{u}|^2dx + \frac{2(\nu-\mu)}{\nu} \int G \p_i u \cdot \na u^i dx \right)_t
+ \frac{3\mu}{4} \| \na \dot{u} \|_{L^2}^2 + \frac{\nu-\mu}{\nu^2} \| \dot{G} \|^2_{L^2} \\
& \le C \| \sqrt{\n} \dot{u} \|^2_{L^2} \| \na u \|^2_{L^2} + \frac{C}{\nu} \| G \|^4_{L^4}
+ \frac{C}{\nu^3} \| P-P(\tilde{\n}) \|^4_{L^4} + C \| \na u \|^4_{L^4} \\
& \le C \| \sqrt{\n} \dot{u} \|^2_{L^2} \left( 1 + \| \na u \|^2_{L^2} + \nu \| \div u \|^2_{L^2} \right)
+ \frac{C}{\nu^3} \| P-P(\tilde{\n}) \|^4_{L^4} \\
& \le C \| \sqrt{\n} \dot{u} \|^2_{L^2} \left( 1 + \| \na u \|^2_{L^2} + \nu \| \div u \|^2_{L^2} \right)
- \frac{C}{\nu} \left( \int (P-P(\tilde{\n}))^3 dx \right)_t.
\ea\ee
Multiplying (\ref{1cp45}) by $\si^2$ yields
\be\la{1cp46}\ba
&\left( \si^2 \int\rho|\dot{u}|^2dx + \frac{2(\nu-\mu)}{\nu} \si^2 \int G \p_i u \cdot \na u^i dx \right)_t
+ \frac{3\mu}{4} \si^2 \| \na \dot{u} \|_{L^2}^2 + \frac{\nu-\mu}{\nu^2} \si^2 \| \dot{G} \|^2_{L^2} \\
& \le 2 \si' \si \left(\int\rho|\dot{u}|^2dx + \frac{2(\nu-\mu)}{\nu} \int G \p_i u \cdot \na u^i dx \right)
+ C \si \| \sqrt{\n} \dot{u} \|^2_{L^2} \\
& \quad - \frac{C}{\nu} \left( \si^2 \int (P-P(\tilde{\n}))^3 dx \right)_t
+ \frac{C}{\nu} \si' \si \left( \int (P-P(\tilde{\n}))^3 dx \right).
\ea\ee
In addition, from (\ref{gw}), (\ref{p2}), and Young's inequality, we deduce that
\be\la{1cp47}\ba
\frac{2(\nu-\mu)}{\nu} \left| \int G \p_i u \cdot \na u^i dx \right|
& \le 2 \left| \int \left( G (\div u)^2 + 2 G \na u^1 \cdot \na^\bot u^2 \right) dx \right| \\
& \le \frac{C}{\nu^2} \| G \|^3_{L^3} + \frac{C}{\nu^2} \| P-P(\tilde{\n}) \|^3_{L^3}
+ C \| \na G \|_{L^2} \| \na u \|^2_{L^2} \\
& \le \frac{C}{\nu^2} \| G \|^2_{L^2} \| \na G \|_{L^2} + \frac{C}{\nu^2} \| P-P(\tilde{\n}) \|^3_{L^3}
+C \| \sqrt{\n} \dot{u} \|_{L^2} \| \na u \|^2_{L^2} \\
& \le \frac{1}{4} \| \sqrt{\n} \dot{u} \|^2_{L^2}
+\frac{C}{\nu^4} \| G \|^4_{L^2} + C \| \na u \|^4_{L^2} + C \\
& \le \frac{1}{4} \| \sqrt{\n} \dot{u} \|^2_{L^2} + C \| \na u \|^4_{L^2} + C.
\ea\ee

Integrating (\ref{1cp46}) over $(0,T)$, and using (\ref{1cpp2}), (\ref{1cp1}), and (\ref{1cp47}),
we get (\ref{1cp04}), thereby completing the proof of Lemma \ref{cpl4}.
\end{proof}

\begin{lemma}\la{cpl5}
There exists a positive constant $\nu_2$ depending only on
$\ga$, $\mu$, $E_0$, $\| \n_0 \|_{L^\infty}$, $\| \na u_0 \|_{L^2}$, and $\tilde{\n}$ such that,
if $(\n,u)$ is a classical solution to \eqref{ns}--\eqref{i3} satisfying \eqref{1cpp2}, then
\be\ba\la{1cp051}
\sup_{0\leq t\leq T} \|\n\|_{L^\infty} 
\leq \frac{3}{2} \left( 1 + \| \n_0 \|_{L^\infty} \right) e^{\frac{1}{\ga} \tilde{\n}^\ga},
\ea\ee
provided $\nu \geq \nu_2$.
\end{lemma}
\begin{proof}
First, rewriting $(\ref{ns})_1$ by using (\ref{gw}) as
\be\ba\la{1cp61}
\pa_t \log\n + u\cdot\na \log\n + \frac{1}{\nu}(\n^\ga - \tilde{\n}^\ga + G) = 0.
\ea\ee
Since $\n^\ga \geq \ga\log\n + 1$, we obtain
\be\ba\la{1cp62}
\frac{d}{ds} \log \n(s) + \frac{\ga}{\nu}\log \n
\le \frac{1}{\nu} \tilde{\n}^\ga + \frac{1}{\nu} \| G \|_{L^\infty}.
\ea\ee
Applying the maximum principle yields
\be\ba\la{1cp63}
\log \n(t)
& \le e^{-\frac{\ga}{\nu}t} \log \n(0)
+ \frac{1}{\nu} \tilde{\n}^\ga \int_0^t e^{-\frac{\ga}{\nu}(t-s)} ds
+ \frac{1}{\nu} \int_0^t e^{-\frac{\ga}{\nu}(t-s)} \| G \|_{L^\infty} ds.
\ea\ee

Next, we give a Poincar\'e-type inequality, whose proof can be found in \cite{LZ2}.
For any $v \in H^1(\rr)$, there exists a positive constant $C$ depending on $\tilde{\n}$ such that
\be\la{1cp640}\ba
\| v \|^2_{L^2} \le C \left( \Vert \sqrt{\rho} v \Vert^2_{L^2}+\| \n-\tilde{\rho} \|^2_{L^2} \| \na v \|^2_{L^2} \right).
\ea\ee
This combined with (\ref{p2}), (\ref{qkjsn}), (\ref{1cp1}) shows that for any $2 \le s<\infty$,
\be\la{1cp64}\ba
\| \na G \|_{L^s} + \| \na \o \|_{L^s} \le C \left( \| \sqrt{\n} \dot{u} \|_{L^2} + \| \na \dot{u} \|_{L^2} \right).
\ea\ee
In view of (\ref{gw}), (\ref{gn11}), (\ref{1cp02}), (\ref{1cp27}), and (\ref{1cp64}), we have
\be\ba\la{1cp65}
\int_0^{\si(t)} \| G \|_{L^\infty} ds
& \le C \int_0^{\si(t)} \| G \|_{L^2}^{3/8} \| \na G \|_{L^5}^{5/8} ds \\
& = C \int_0^{\si(t)} \| G \|_{L^2} ^{3/8} \left(\si^2\| \na G \|^{2}_{L^5} \right)^{5/16} \si^{ -5/8 } ds \\
& \le C \nu ^{3/8} \int_0^{\si(t)} \left( \si^2 \| \sqrt{\n} \dot u \|^{2}_{L^2} + \si^2 \| \na \dot u \|^{2}_{L^2} \right)^{5/16}
\si^{-5/8} ds \\
& \le C \nu^{3/8} \left( \int_0^1 \si^{-10/11} ds \right)^{11/16} \\
& \le C \nu^{3/8},
\ea\ee
and
\be\ba\la{1cp66}
\int_{\si(t)}^t \| G \|^3_{L^\infty} ds
& \le C \int_{\si(t)}^t \| G \|_{L^2} \| \na G \|^2_{L^4} ds \\
& \le C \nu^{1/2} \int_{\si(t)}^t \left( \| \sqrt{\n} \dot{u} \|^2_{L^2} + \| \na \dot{u} \|^2_{L^2} \right) ds \\
& \le C \nu^{1/2}.
\ea\ee
From (\ref{1cp65}), (\ref{1cp66}), and H\"older's inequality, we conclude that
\be\la{1cp67}\ba
\int_0^t e^{-\frac{\ga}{\nu}(t-s)} \| G \|_{L^\infty} ds
& = \int_0^{\si(t)} e^{-\frac{\ga}{\nu}(t-s)} \| G \|_{L^\infty} ds
+ \int_{\si(t)}^t e^{-\frac{\ga}{\nu}(t-s)} \| G \|_{L^\infty} ds \\
& \le C \int_0^{\si(t)} \| G \|_{L^\infty} ds
+ \left( \int_{\si(t)}^t e^{-\frac{3\ga}{2\nu}(t-s)} ds \right)^{\frac{2}{3}}
\left( \int_{\si(t)}^t \| G \|^3_{L^\infty} ds \right)^{\frac{1}{3}} \\
& \le C \nu^{\frac{5}{6}}.
\ea\ee
Putting (\ref{1cp67}) into (\ref{1cp63}) gives
\be\ba\la{1cp68}
\log \n(t)
& \le e^{-\frac{\ga}{\nu}t} \log \left( 1 + \| \n_0 \|_{L^\infty} \right)
+ \frac{1}{\ga} \tilde{\n}^\ga ( 1-e^{-\frac{\ga}{\nu}t} ) + C \nu^{-\frac{1}{6}} \\
& \le \log \left( 1 + \| \n_0 \|_{L^\infty} \right)
+ \frac{1}{\ga} \tilde{\n}^\ga + M_2 \nu^{-\frac{1}{6}},
\ea\ee
where $M_2$ is a positive constant depending only on 
$\tilde{\n}$, $\ga$, $\mu$, $\|\n_0\|_{L^\infty}$, $E_0$, and $\|\na u_0\|_{L^2}$, but is independent of $T$ and $\nu$.

Finally, set
\be\ba\la{1cp69}
\nu_2 \triangleq \max \left\{ \hat{\nu}_2, \left( \frac{M_2}{\log \frac{3}{2}} \right)^6 \right\},
\ea\ee
with $\hat{\nu}_2$ given in (\ref{1cp325}).
Then, when $\nu\geq\nu_2$, we obtain (\ref{1cp051}) and complete the proof of Lemma \ref{cpl5}.
\end{proof}

\section{A Priori Estimates \uppercase\expandafter{\romannumeral2}: Higher Order Estimates}

To extend the local classical solution globally in time,
in this section we establish the higher-order estimates, which are similar to those in \cite{HL,HLX2,LX2,LLL}.

\subsection{Far-field density is vacuum}

In this subsection, we assume that $\nu \ge \nu_1$ where $\nu_1$ is determined in Proposition \ref{l5},
and let $(\n,u)$ be a classical solution of (\ref{ns})--(\ref{i3}) on $\rr \times (0,T]$ satisfying (\ref{cp051}).

First, we assume that the initial data $(\n_0,u_0)$ 
satisfying $\eqref{ssol1}_1$.
\begin{lemma}\la{s21}
There is a positive constant $C$ depending only on  
$T,\ \ga,\ \mu,\ \nu,\ E_0,\ \| {\bar{x}}^a \rho_0 \|_{L^1}$, $\|\n_0\|_{L^\infty}$ and
$\| \na u_0 \|_{L^2}$, such that
\be\la{s411} \ba
\sup_{0\le t\le T} \si \int\n|\dot u|^2dx+\int_0^{ T} \si \|\na\dot u\|^2_{L^2}dt\le C.
\ea\ee	
\end{lemma}
\begin{proof}
First, we deduce from (\ref{cp051}), (\ref{gw}), (\ref{cp1}) and (\ref{cp031}) that
\be\la{s418} \ba
\sup_{0\le t\le T} \left( \| \n \|_{L^\infty} + \| \na u \|_{L^2} \right) 
+ \int_0^T \left( \| \na u\|^2_{L^2}+ \| \n^{1/2} \dot{u} \|^2_{L^2} \right)dt \le C.
\ea \ee
Moreover, it follows from (\ref{pd118}) and (\ref{s418}) that
\be\la{s417}\ba
&\left(\int\rho|\dot{u}|^2dx + \frac{2(\nu-\mu)}{\nu} \int G \p_i u \cdot \na u^i dx \right)_t
+ \frac{3\mu}{4} \| \na \dot{u} \|_{L^2}^2 \\
& \le C \| \na u \|^2_{L^2} + C \| \sqrt{\n} \dot{u} \|^2_{L^2}
+ \frac{C}{\nu^2} \| P-P(\tilde{\n}) \|^4_{L^4} + C \| \na u \|^4_{L^4}.
\ea\ee
Combining (\ref{p2}), (\ref{s418}), (\ref{dc1}) and (\ref{gn11}) implies
\be\la{s419} \ba
\int_0^{T}\|\na u\|^4_{L^4}dt & \le C\int_0^{ T}\|\div u\|^4_{L^4} +\|\o \|^4_{L^4}dt\\
&\le C\int_0^{T} \|G\|^2_{L^2}\|\nabla G\|^2_{L^2}
+ \| P \|^4_{L^4} + \|\o \|^2_{L^2}\|\nabla \o \|^2_{L^2}dt \\
&\le C + C\int_0^{T} \|\nabla G\|^2_{L^2} + \|\nabla \o \|^2_{L^2}dt \\
&\le C.
\ea\ee
Multiplying (\ref{s417}) by $\si$ and integrating the resulting equation over $(0,T)$,
after using (\ref{pd120}), (\ref{s418}) and (\ref{s419}), we obtain (\ref{s411}).
\end{proof}

\begin{lemma}\la{s22}
For any $2 < p < \infty$, there exists a positive constant $C$ depending only on 
$T,\ q,\ \ga,\ \mu,\ \nu,\ E_0,\ N_0$, $\| {\bar{x}}^a \rho_0 \|_{L^1}$, 
$\|\n_0\|_{L^\infty}$ and $\| \na u_0 \|_{L^2}$, such that
\be\la{s431} \ba
&\sup_{0\le t\le T} \left( \| \n \|_{H^1 \cap W^{1,q}} + \| \na u \|_{L^2} + t \| \na^2 u \|^2_{L^2} \right) \\
& + \int_0^T \left( \|\nabla^2 u\|^{2}_{L^2}+\|\nabla^2 u\|^{(q+1)/q}_{L^q}+t \|\nabla^2 u\|_{L^q}^2 \right) dt\le C.
\ea\ee
\end{lemma}
\begin{proof}
First, for any $p \in [2,q]$, we deduce from $(\ref{ns})_1$ that $|\na \n|^p$ satisfies
\be\la{s432}\ba
& (|\nabla\n|^p)_t + \text{div}(|\nabla\n|^p u)+ (p-1)|\nabla\n|^p\text{div}u  \\
&+ p|\nabla\n|^{p-2} \p_i\n \p_i u^j \p_j\n +
p \n |\nabla\n|^{p-2}\p_i\n  \p_i \text{div}u = 0.
\ea\ee
Integrating (\ref{s432}) over $\rr$ gives
\be\la{s433} \ba
\frac{d}{dt} \norm[L^p]{\nabla\n}  
&\le C \norm[L^{\infty}]{\nabla u}  \norm[L^p]{\nabla\n} +C\|\na^2 u\|_{L^p} \\ 
&\le C(1+\norm[L^{\infty}]{\nabla u} ) \norm[L^p]{\nabla\n}+C \|\n\dot u\|_{L^p}, \ea\ee 
where we have used the following estimate:
\be\la{s434}\ba
\|\na^2 u\|_{L^p}
&\le C(\|\na \div u\|_{L^p}+\|\na \o\|_{L^p}) \\
&\le C (\| \nabla G\|_{L^p}+ \|\nabla P \|_{L^p})+ C\|\na \o\|_{L^p} \\
&\le C\|\n\dot u\|_{L^p} + C  \|\nabla \n \|_{L^p},
\ea\ee
owing to (\ref{p2}) and (\ref{dc1}).

In addition, with the help of (\ref{gn11}) and (\ref{s418}), we have
\be\la{s435}\ba 
\|\div u\|_{L^\infty}+\|\o\|_{L^\infty} 
& \le C + C \|\na G\|_{L^q}^{q/(2(q-1))} +C \|\na \o\|_{L^q}^{q/(2(q-1))} \\
& \le C + C \|\n\dot u\|_{L^q}^{q/(2(q-1))},
\ea\ee 
which together with Lemma \ref{bkm} and (\ref{s434}) implies
\be\la{s436}\ba   \|\na u\|_{L^\infty }  
&\le C\left(\|{\rm div}u\|_{L^\infty }+ \|\o\|_{L^\infty } \right)\log(e+\|\na^2 u\|_{L^q}) +C\|\na u\|_{L^2} +C \\
&\le C\left(1+\|\n\dot u\|_{L^q}^{q/(2(q-1))}\right)\log(e+\|\rho \dot u\|_{L^q} +\|\na \rho\|_{L^q}) +C\\
&\le C\left(1+\|\n\dot u\|_{L^q} \right) \log(e+ \|\na \rho\|_{L^q}).
\ea\ee
Then, taking $p=q$ in (\ref{s433}) and applying, we derive
\be\la{s437}\ba
\frac{d}{dt} \log(e+ \|\na \rho\|_{L^q})
\le C\left(1+\|\n\dot u\|_{L^q} \right) \log(e+ \|\na \rho\|_{L^q}).
\ea\ee 
On the other hand, by virtue of (\ref{esnr2}) and H\"{o}lder's inequality, it holds that
\be\la{s422}\ba
\| \rho \dot u\|_{L^q} 
& \le C\| \rho \dot u\|_{L^2}^{2(q-1)/(q^2-2)}
\| \n \dot{u} \|_{L^{q^2}}^{q(q-2)/(q^2-2)} \\ 
& \le C\| \rho \dot u\|_{L^2}^{2(q-1)/(q^2-2)} 
\left( \| \sqrt{\n} \dot u\|_{L^2}+\| \na \dot u\|_{L^2} \right)^{q(q-2)/(q^2-2)} \\ 
& \le C\| \sqrt{\n}  \dot u\|_{L^2} + C\| \sqrt{\n} \dot u\|_{L^2}^{2(q-1)/(q^2-2)}
\|\na \dot u\|_{L^2}^{q(q-2)/(q^2-2)},
\ea\ee
which together with (\ref{s411}), (\ref{s422}) and Young's inequality yields
\be\la{s423} \ba
& \int_0^T \left( \|\rho \dot u\|^{1+1 /q}_{L^q}+t\| \n \dot u\|^2_{L^q} \right) dt \\
& \le C+C \int_0^T\left( \| \sqrt{\n}  \dot u\|_{L^2}^2 +  t\|\na \dot u\|_{L^2}^2+ 
t^{-(q^3-q^2-2q)/(q^3-q^2-2q+2)} \right)dt \\ 
& \le C.
\ea\ee

Consequently, by using (\ref{s437}), (\ref{s423}) and Gr\"onwall's inequality, we derive
\be\la{s439}\ba
\sup\limits_{0\le t\le T} \|\nabla
\rho\|_{L^q}\le  C,
\ea\ee
which together with (\ref{s411}), (\ref{s434}), (\ref{s423}) and (\ref{s439}) results in
\be\la{s4310}\ba 
\int_0^T \left( \|\nabla^2 u\|^{(q+1)/q}_{L^q}+t \|\nabla^2 u\|_{L^q}^2 \right) dt \le C.
\ea\ee

Moreover, choosing $p=2$ in (\ref{s433}) and applying (\ref{s411}), (\ref{s418}), (\ref{s423}) and Gr\"onwall's inequality, we obtain
\be\la{}\ba
\sup\limits_{0\le t\le T} \left( \|\nabla \rho\|_{L^2} + t \| \na^2 u \|^2_{L^2} \right)
+ \int_0^T \| \na^2 u \|^2_{L^2} dt
\le C,
\ea\ee
which together with (\ref{s439}), (\ref{s4310}) and (\ref{s418}) gives (\ref{s431}), 
and we finish the proof of Lemma \ref{s22}.
\end{proof}

\begin{lemma}\la{s23}
There exists a constant $C $ depending only on $\mu,\ \ga,\ T,\ N_0,\ E_0,\ q$, 
$\|\n_0\|_{L^\infty}$, $\| \na u_0 \|_{L^2}$,
$\| {\bar{x}}^a \rho_0 \|_{L^1}$, and $\|\na(\bar x^a\n_0)\|_{L^2\cap L^q}$, such that
\be\la{s05}\ba
&\sup_{0\le t\le T} \|  \bar x^a\n \|_{L^1\cap H^1\cap W^{1,q}}  \le C.
\ea\ee
\end{lemma}
\begin{proof}
First, multiplying $(\ref{ns})_1$ by ${\bar{x}}^a$ and integrating over $\rr$, we derive after using (\ref{cp1}) that
\be\la{cp12}\ba
\frac{d}{dt} \int \n {\bar{x}}^a dx 
& \le C \int \n |u| {\bar{x}}^{a-1} \log^2(e+|x|^2) dx \\
& \le C \left( \int \n {\bar{x}}^{2a-2} \log^4(e+|x|^2) dx \right)^{\frac{1}{2}}
\left( \int \n |u|^2 dx \right)^{\frac{1}{2}} \\
& \le C \left( \int \n {\bar{x}}^a dx \right)^{\frac{1}{2}},
\ea\ee
which together with Gr\"onwall's inequality gives
\be\ba\la{cp13}
\sup_{0\leq t\leq T} \int \n {\bar{x}}^a dx \leq C.
\ea\ee
From (\ref{WPE1}), (\ref{pt1}), and (\ref{s431}), we deduce that for any $\ve\in(0,1]$ and any $s>2$,
\be\la{s51}\ba
\|u\bar x^{-\ve}\|_{L^{s/\ve}}\le C(\ve,s).
\ea\ee
Moreover, according to $(\ref{ns})_1$, we derive that $ v\triangleq\n\bar x^a$ satisfies
\bnn\ba
v_t+u\cdot\na v-a vu\cdot\na \log \bar x+v\div u=0,
\ea\enn
which together with (\ref{s51}) and H\"older's inequality implies that for any $p\in [2,q]$,
\be\la{s52}\ba 
(\|\na v\|_{L^p} )_t
& \le C(1+\|\na u\|_{L^\infty}+\|u\cdot \na \log \bar x\|_{L^\infty}) \|\na v\|_{L^p} \\
&\quad +C\|v\|_{L^\infty}\left( \||\na u||\na\log \bar x|\|_{L^p}+\||  u||\na^2\log \bar x|\|_{L^p}+\| \na^2 u \|_{L^p}\right)\\
& \le C(1 +\|\na u\|_{W^{1,q}})  \|\na v\|_{L^p} \\
& \quad + C \|v\|_{L^\infty} \left( \|\na u\|_{L^p} 
+ \| u \bar x^{-2/5} \|_{L^{4p}} \|\bar x^{-3/2}\|_{L^{4p/3}} + \|\na^2 u\|_{L^p}\right) \\
& \le C(1 +\|\na^2u\|_{L^p}+\|\na u\|_{W^{1,q}})(1+ \|\na v\|_{L^p}+\|\na v\|_{L^q}). \ea\ee
Taking $p=q$ in (\ref{s52}) and applying (\ref{s431}) and Gr\"onwall's inequality, we obtain
\be\la{s53}\ba
\sup\limits_{0\le t\le T}\|\na (\n \bar x^a)\|_{L^q} \le C.
\ea\ee

In addition, choosing $p=2$ in (\ref{s52}) and combining (\ref{s431}), (\ref{s53}) and Gr\"onwall's inequality yields
\bnn
\sup\limits_{0\le t\le T}\|\na(\n \bar x^a)\|_{L^2 } \le C,
\enn
which together with (\ref{s53}) and (\ref{cp13}) implies (\ref{s05}).
This completes the proof of Lemma \ref{s23}.
\end{proof}

\begin{lemma}\la{s24}  
There exists a constant $C $ depending only on $\mu,\ \ga,\ T,\ N_0,\ E_0,\ q$, 
$\|\n_0\|_{L^\infty}$, $\| \na u_0 \|_{L^2}$,
$\| {\bar{x}}^a \rho_0 \|_{L^1}$, and $\|\na(\bar x^a\n_0)\|_{L^2\cap L^q}$, such that
\be\la{s06}\ba
\sup_{0\leq t\leq T }t\left(\| \sqrt{\n} u_t\|^2_{L^2}+\|\na u\|^2_{H^1} \right)+\int_0^Tt\|\na u_t\|_{L^2}^2dt\le C.
\ea\ee
\end{lemma}

\begin{proof}
First, it follows from (\ref{WPE1}), (\ref{cp4}) and (\ref{s431}) that
for any $\eta\in(0,1]$ and any $s>2$,
\be\la{s61}\ba
\|\n^\eta u \|_{L^{s/\eta}}+ \|u\bar x^{-\eta}\|_{L^{s/\eta}}\le C(\eta,s).
\ea\ee 

Then, differentiating $(\ref{ns})_2$ with respect to $t$, we derive
\be\la{s620}\ba
& \n u_{tt}+\n u\cdot \na u_t-\mu\Delta u_t-( \mu+\lm)\na  \div u_t  \\ 
&= -\n_t(u_t+u\cdot\na u)-\n u_t\cdot\na u -\na P_t.
\ea\ee
Multiplying (\ref{s620}) by $u_t$ and integrating the resulting equation over $\rr$, 
then using $(\ref{ns})_1$ yields
\be\la{s62}\ba
& \frac{1}{2}\frac{d}{dt} \int \n |u_t|^2dx+\int \left(\mu|\na u_t|^2+( \mu+\lm)(\div u_t)^2  \right)dx \\
& =-2\int \n u \cdot \na  u_t\cdot u_tdx  -\int \n u \cdot\na (u\cdot\na u\cdot u_t)dx\\
& \quad-\int \n u_t \cdot\na u \cdot  u_tdx
+\int P_{t}\div u_{t} dx \\
& \quad \triangleq I_1+I_2+I_3+I_4.
\ea\ee
We conclude from (\ref{s61}), (\ref{s431}), (\ref{gn11}) and H\"older's inequality that
\be\la{s63}\ba
I_1+I_2 
& \le C \int  \n |u| \left( |\na u_t| |u_t| + |\na u|^2 |u_t| + |u| |u_t| |\na^2 u| + |u| |\na u| |\na u_t| \right) dx \\
& \le C \| \sqrt{\n} u\|_{L^{6}}\| \sqrt{\n} u_{t} \|_{L^{2}}^{1/2} \| \sqrt{\n} u_{t}\|_{L^{6}}^{1/2}\left(\| \na u_{t}\|_{L^{2}}+\| \na u\|_{L^{4}}^{2} \right) \\
&\quad + C\|\n^{1/4} u \|_{L^{12}}^{2}\| \sqrt{\n} u_{t}\|_{L^{2}}^{1/2} \| \sqrt{\n} u_{t}\|_{L^{6}}^{1/2} \| \na^{2} u \|_{L^{2}} 
+ C \| \sqrt{\n} u\|_{L^{8}}^{2}\|\na u\|_{L^{4}} \| \na u_{t}\|_{L^{2}} \\
& \le C  \| \sqrt{\n} u_{t}\|_{L^{2}}^{1/2}\left( \| \sqrt{\n} u_{t}\|_{L^{2}} +\| \na u_{t}\|_{L^{2}}\right)^{1/2}\left( \| \na u_{t}\|_{L^{2}} +  \| \na^{2} u \|_{L^{2}}+1\right) \\
& \quad + C \| \na u_{t}\|_{L^{2}} \left( 1 +  \| \na^{2} u \|_{L^{2}} \right) \\
&\le  \de\| \na u_{t}\|_{L^{2}}^{2}+C(\de)  \left(\| \na^{2} u \|_{L^{2}}^{2} +  \|\n^{1/2} u_{t}\|_{L^{2}}^{2}+1\right).
\ea\ee
where in the third inequality, we have used the following estimate
\be\la{s64}\ba
\| \sqrt{\n} u_t\|_{L^6} \le C \| \sqrt{\n} u_t\|_{L^2}+C \|\na u_t\|_{L^2},
\ea\ee
due to (\ref{esnr2}) and (\ref{cp45}).

In addition, by virtue of (\ref{s64}), (\ref{s431}) and Young's inequality, we have
\be\la{s65}\ba
I_3 + I_4 & \le C \int \n |u_t|^{2}|\na u | + |P_t| |\na u_t| dx \\
& \le C \| \na u\|_{L^{2}} \| \sqrt{\n} u_{t}\|_{L^{6}}^{3/2}
\| \sqrt{\n} u_{t}\|_{L^{2}}^{1/2} + \| P_t \|_{L^2} \| \na u_t \|_{L^2} \\
& \le \de \| \na u_{t}\|_{L^{2}}^{2}+C(\de) \left(\| \na^{2} u \|_{L^{2}}^{2} +
\| \sqrt{\n} u_{t}\|_{L^{2}}^{2}+1\right),
\ea\ee
where in the last inequality we have used the following fact:
\be\la{s66}\ba 
\|P_t\|_{L^2 } &\le C\|\bar x^{-a} u\|_{L^{2q/(q-2)}}\|\n\|_{L^\infty}^{\ga-1}\|\bar x^a \na \n\|_{  L^q}+C\|\na u\|_{L^2 }\le C,
\ea\ee
owing to $(\ref{ns})_1$, (\ref{s61}) and (\ref{s431}).

Moreover, it follows from (\ref{gn11}), (\ref{s61}) and H\"older's inequality that
\be\la{s67}\ba
\| \sqrt{\n} u_t \|^2_{L^2} 
& \le C \left( \| \sqrt{\n} \dot{u} \|^2_{L^2}
+ \| \sqrt{\n} u \cdot \na u \|^2_{L^2} \right) \\
& \le C \left( \| \sqrt{\n} \dot{u} \|^2_{L^2}
+ \| \sqrt{\n} u \|^2_{L^6} \| \na u \|^2_{L^3} \right) \\
& \le C \left( \| \sqrt{\n} \dot{u} \|^2_{L^2}
+ \| \na^2 u \|^2_{L^2} \right).
\ea\ee

Substituting (\ref{s63}) and (\ref{s65}) into (\ref{s62}), taking $\de$ suitably small, and applying (\ref{s67}), we derive
\be\la{s68}\ba
\frac{d}{dt} \int \n |u_t|^2dx+\mu\int   |\na u_t|^2 dx
\le C \left(\| \na^{2} u \|_{L^{2}}^{2} + \| \sqrt{\n} \dot{u} \|^2_{L^2} + 1 \right).
\ea\ee

Multiplying (\ref{s68}) by $t$, we derive
\be\la{s69}\ba
\frac{d}{dt} \left( t \int \n |u_t|^2 dx \right) + \mu t \int |\na u_t|^2 dx
\le C \left(\| \na^{2} u \|_{L^{2}}^{2} + \| \sqrt{\n} \dot{u} \|^2_{L^2} + 1 \right).
\ea\ee
Integrating (\ref{s69}) over $(0,T)$ implies (\ref{s06}) and finished the proof of Lemma \ref{s24}.
\end{proof}

From now on, we assume that the initial data $(\n_0,u_0)$ satisfy $(\ref{csol1})_1$ and the compatibility condition (\ref{csol2}).
\begin{lemma}\la{c21}
There exists a constant $C $ depending only on $\mu,\ \ga,\ T,\ N_0,\ E_0,\ q$, 
$\|\n_0\|_{L^\infty}$, $\| \na u_0 \|_{H^1}$,
$\| {\bar{x}}^a \rho_0 \|_{L^1}$, $\|\na(\bar x^a\n_0)\|_{L^2\cap L^q}$
and $\| g \|_{L^2}$, such that
\be\la{c01}\ba
\sup_{0\leq t\leq T} \left(\| \sqrt{\n} u_t\|^2_{L^2}+\|\na u\|^2_{H^1} \right)
+ \int_0^T \|\na u_t\|_{L^2}^2 dt \le C.
\ea\ee
\end{lemma}
\begin{proof}
First, taking into account on the compatibility condition (\ref{csol2}), we define
\be\la{c413}\ba
\sqrt{\n} \dot{u}(x,t=0)=g(x).
\ea\ee
Integrating (\ref{s417}) over $(0,T)$
and using (\ref{s418}) and (\ref{s419}), we derive
\be\la{c411}\ba
\sup_{0\le t\le T}
\int\n|\dot u|^2dx
+\int_0^{ T}\int|\na\dot u|^2dxdt\le C,
\ea\ee
which together with (\ref{s431}), (\ref{s434}) and (\ref{s422}) implies
\be\la{c0411}\ba
\int_0^{T} \| \na^2 u \|^2_{L^2} + \| \na^2 u \|^2_{L^q} dxdt \le C.
\ea\ee
Then, integrating (\ref{s68}) over $(0,T)$, with the help of (\ref{c0411}),
we obtain (\ref{c01}) and finish the proof of Lemma \ref{c21}.
\end{proof}

In order to extend the local classical solution, we need the following higher-order estimates.
These estimates can be shown using similar arguments as in \cite{LLL}, so the proofs are omitted.
\begin{lemma}\la{c26}
There exists a positive constant C depending only on $T,\ N_0,\ \mu,\ \nu,\ \ga,\ E_0$,
$\|\n_0\|_{L^\infty}$, $\| \na u_0 \|_{H^1}$,
$\| {\bar{x}}^a \rho_0 \|_{L^1}$,
$\|\na ( \bar x^a \rho_0 ) \|_{L^2 \cap L^{q}}$, $\|\bar x^{\de_0} \na^2 \n \|_{L^2}$,
$ \|\bar x^{\de_0} \na^2 P \|_{L^2}$ and $\|g\|_{L^2}$, such that
\be\la{c421}\ba
\sup_{0\le t\le T} \left( \| \bar x^{\de_0} \na^2 \n \|_{L^2}
+ \| \bar x^{\de_0} \na^2 P \|_{L^2} \right) 
\le C,
\ea\ee

\be\la{c442}\ba
\sup\limits_{0\le t\le T} t \| \na u_t\|^2_{L^2}
+ \int_0^T t\left(\|\n^{1/2}u_{tt}\|^2_{L^2}+\| \na^2 u_t\|^2_{L^2}\right)dt
\le C,
\ea\ee

\be\la{c451} \ba
\sup_{0\leq t\leq T}\left(\|\nabla^2 \n\|_{L^q } +\|\nabla^2 P  \|_{L^q }\right) \leq C,
\ea\ee

\be\la{c461}\ba
& \sup_{0\leq t\leq T} t \left(\| \sqrt{\n} u_{tt}\|_{L^2}+  \|\na^3 u \|_{L^2 \cap L^q}
+ \| \na u_t \|_{H^1} + \| \na^2 (\n u) \|_{L^{(q+2)/2}} \right) \\
& \quad + \int_{0}^T t^2 \left( \|\nabla u_{tt}\|_{L^2}^2 + \| u_{tt} \bar x^{-1} \|_{L^2}^2 \right) dt
\leq C.
\ea\ee
\end{lemma}

\subsection{Far-field density is away from vacuum}

In this subsection, we assume that $\nu \ge \nu_2$ where $\nu_2$ is determined in Lemma \ref{cpl5},
and let $(\n,u)$ be a classical solution of (\ref{ns})--(\ref{i3}) on $\rr \times (0,T]$ satisfying (\ref{1cpp2}).

First, we assume that the initial data $(\n_0,u_0)$ satisfying $(\ref{ssol1})_2$.
\begin{lemma}\la{lsn21}
There is a positive constant $C$ depending only on  
$T,\ \ga,\ \mu,\ \nu,\ E_0$, $\| \rho_0 \|_{L^\infty}$, $\tilde{\n}$ and
$\| \na u_0 \|_{L^2}$, such that
\be\la{sn411} \ba
\sup_{0\le t\le T} \si \int\n|\dot u|^2dx+\int_0^{ T} \si \|\na\dot u\|^2_{L^2}dt\le C.
\ea\ee	
\end{lemma}

\begin{proof}
First, it follows from (\ref{cp1}) and (\ref{1cpp2}) that
\be\la{sn418}\ba
\sup_{0\le t\le T} \| \na u\|_{L^2} + \int_0^T \left( \| \na u\|^2_{L^2}+ \| \n^{1/2} \dot{u} \|^2_{L^2} \right)dt \le C.
\ea\ee
Then, analogue to the proof of Lemma \ref{s21}, we have
\be\la{sn417}\ba
&\left(\int\rho|\dot{u}|^2dx + \frac{2(\nu-\mu)}{\nu} \int G \p_i u \cdot \na u^i dx \right)_t
+ \frac{3\mu}{4} \| \na \dot{u} \|_{L^2}^2 \\
& \le C \| \na u \|^2_{L^2} + C \| \sqrt{\n} \dot{u} \|^2_{L^2}
+ \frac{C}{\nu^2} \| P-P(\tilde{\n}) \|^4_{L^4} + C \| \na u \|^4_{L^4}.
\ea\ee
In addition, by virtue of (\ref{p2}), (\ref{sn418}), (\ref{dc1}) and (\ref{gn11}), it holds that
\be\la{sn419}\ba
\int_0^{T}\|\na u\|^4_{L^4}dt & \le C\int_0^{ T}\|\div u\|^4_{L^4} +\|\o \|^4_{L^4}dt\\
&\le C\int_0^{T} \|G\|^2_{L^2}\|\nabla G\|^2_{L^2}
+ \| P-P(\tilde{\rho}) \|^4_{L^4} + \|\o \|^2_{L^2}\|\nabla \o \|^2_{L^2}dt \\
&\le C + C\int_0^{T} \|\nabla G\|^2_{L^2} + \|\nabla \o \|^2_{L^2}dt \\
&\le C.
\ea\ee
Multiplying (\ref{sn417}) by $\si$ and integrating the resulting equation over $(0,T)$,
with the help of (\ref{pd120}), (\ref{sn418}) and (\ref{sn419}), we derive (\ref{sn411}).
\end{proof}

\begin{lemma}\la{lsn22}
For any $2 < p < \infty$, there exists a positive constant $C$ depending only on 
$T,\ p,\ \ga,\ \mu,\ \nu,\ E_0,\ \| \rho_0 \|_{L^\infty}$, $\| \na \n_0 \|_{L^2 \cap L^q}$, $\tilde{\n}$ and
$\| \na u_0 \|_{L^2}$, such that
\be\la{sn431}\ba
&\sup_{0\le t\le T} \left( \| \n - \tilde{\rho} \|_{H^1 \cap W^{1,q}} + \| u \|_{H^1} 
+ t \| \na^2 u \|^2_{L^2} + t \| \sqrt{\n} u_t \|^2_{L^2} + \| \n_t \|_{L^2} \right) \\
& + \int_0^T \left( \| u \|^{2}_{H^2}+\|\nabla^2 u\|^{(q+1)/q}_{L^q}+t \|\nabla^2 u\|_{L^q}^2 
+ \| \sqrt{\n} u_t \|^2_{L^2} + t \| u_t \|^2_{H^1} \right) dt
\le C.
\ea\ee
\end{lemma}
\begin{proof}
First, following an argument analogous to the proof of Lemma \ref{s22} and employing estimates (\ref{LZ02}) and (\ref{1cp27}), we obtain
\be\la{sn432}\ba
&\sup_{0\le t\le T} \left( \| \n - \tilde{\rho} \|_{H^1 \cap W^{1,q}} + \| u \|_{H^1} 
+ t \| \na^2 u \|^2_{L^2} \right) \\
& + \int_0^T \left( \| u \|^{2}_{H^2}+\|\nabla^2 u\|^{(q+1)/q}_{L^q}+t \|\nabla^2 u\|_{L^q}^2 \right) dt
\le C.
\ea\ee
Then, it follows from $(\ref{ns})_1$, (\ref{LZ02}), (\ref{sn432}) and H\"older's inequality, we obtain
\be\la{sn4311}\ba
\| \n_t\|_{L^2}\le
C\|u\|_{L^{2q/(q-2)}}\|\nabla \n \|_{L^q}+C\|\n\|_{L^\infty} \|\nabla u\|_{L^2} \le C.
\ea\ee

In addition, we deduce from (\ref{sn432}), (\ref{gn11}) and H\"older's inequality that
\be\la{sn4312}\ba 
\int\rho|u_t|^2dx 
&\le \int\rho| \dot u |^2dx+\int \n |u\cdot\na u|^2dx \\
&\le \int\rho| \dot u |^2dx+C \| u \|_{L^4}^2 \| \na u\|_{L^4}^2 \\ 
&\le \int\rho| \dot u |^2dx+C\| \na^2 u\|_{L^2}^2.\ea \ee 
Similarly, we also have 
\be\la{sn4313}\ba 
\|\nabla u_t\|_{L^2}^2 
&\le \| \nabla \dot u \|_{L^2}^2+ \| \nabla(u\cdot\nabla u)\|_{L^2}^2  \\ 
&\le \|\nabla \dot u\|_{L^2}^2+ \|u\|_{L^{2q/(q-2)}}^2\|\nabla^2u \|_{L^q}^2+ \| \nabla u \|_{L^4}^4 \\ 
&\le \|\nabla \dot u\|_{L^2}^2+C\|\nabla^2u \|_{L^q}^2+ \| \nabla u \|_{L^4}^4,
\ea\ee 
which together with (\ref{LZ02}), (\ref{sn411}), (\ref{sn432}), (\ref{sn4312}) and (\ref{sn4313}) yields
\be\la{sn4314}\ba
\sup_{0\le t\le T} t \| \sqrt{\n} u_t \|^2_{L^2} 
+ \int_0^T \| \sqrt{\n} u_t \|^2_{L^2} + t\|u_t\|_{H^1}^2 dt
\le C.
\ea\ee
Combining this with (\ref{sn432}) leads to (\ref{sn431}), and we finish the proof of Lemma \ref{lsn22}.
\end{proof}

Thereafter, we assume the initial data $(\n_0,u_0)$ satisfy both the regularity condition $(\ref{csol1})_2$
and the compatibility condition (\ref{csol2}).
\begin{lemma}\la{lcn21}
There exists a positive constant $C$ depending only on  
$T,\ p,\ \ga,\ \mu,\ \nu,\ E_0,\ \| \rho_0 \|_{L^\infty}$, $\| \na \n_0 \|_{L^2 \cap L^q}$, $\tilde{\n}$,
$\| \na u_0 \|_{H^1}$ and $\| g \|_{L^2}$, such that
\be\la{cn4301}\ba
\sup_{0\le t\le T} \left( \| u \|^2_{H^2} + \| \sqrt{\n} u_t \|^2_{L^2} \right)
+ \int_0^T \left( \|\nabla^2 u\|_{L^q}^2 + \| u_t \|^2_{H^1} \right) dt
\le C.
\ea\ee
\end{lemma}
\begin{proof}
In view of the compatibility condition (\ref{csol2}), we are able to define
\be\la{cn413}\ba
\sqrt{\n} \dot{u}(x,t=0)=g(x).
\ea\ee
Then, integrating (\ref{sn417}) over $(0,T)$ and applying (\ref{sn418}) and (\ref{sn419}), we derive
\be\la{cn411}\ba
\sup_{0\le t\le T}
\int\n|\dot u|^2dx
+\int_0^{T}\int|\na\dot u|^2dxdt\le C,
\ea\ee
which together with (\ref{s434}), (\ref{s422}), (\ref{sn4312}) and (\ref{sn4313}) yields (\ref{cn4301}),
and we finish the proof of Lemma \ref{lsn22}.
\end{proof}

To extend the local classical solution, we require the following higher-order estimates.
Since the proofs of these estimates are similar to those in \cite{HLX2}, we omit the proofs.
\begin{lemma}\la{lcn26}
There exists a positive constant C depending only on $T,\ \mu,\ \nu,\ \ga,\ E_0$, 
$\tilde{\rho}$, $\| \rho_0-\tilde{\rho} \|_{W^{2,q}}$, $\| \na^2 P(\rho_0) \|_{L^q}$, $\| \na u_0 \|_{H^1}$ and $\|g\|_{L^2}$, such that
\be\la{cn431}\ba 
& \sup_{t\in[0,T]} \left(\norm[H^2]{\rho-\tilde{\rho}} 
+ \norm[H^2]{ P(\rho)-P(\tilde{\rho}) }+ \|\n_t\|_{H^1}+\|P_t\|_{H^1} \right) \\
&\quad + \int_0^T\left(\| \na^3 u\|^2_{L^2}+\|\n_{tt}\|_{L^2}^2 +\|P_{tt}\|_{L^2}^2\right)dt \le C,
\ea\ee

\be\la{cn442}\ba
\sup\limits_{0\le t\le T} t \left(\| \na u_t\|^2_{L^2}+\| \na^3 u\|^2_{L^2} \right) 
+ \int_0^T t\left(\| \sqrt{\n} u_{tt}\|^2_{L^2}+\| \na^2 u_t\|^2_{L^2}\right)dt
\le C,
\ea\ee

\be\la{cn451} \ba
\sup_{0\leq t\leq T}\left(\|\nabla^2 \n\|_{L^q } +\|\nabla^2 P  \|_{L^q }\right) \leq C,
\ea\ee

\be\la{cn461}\ba
\sup_{0\leq t\leq T} t\left( \| \sqrt{\n} u_{tt}\|_{L^2}+  \|\na^3 u \|_{L^q} + \|\na^2
u_t \|_{L^2}  \right) +\int_{0}^T  t^2 \|\nabla u_{tt}\|_{L^2}^2 dt\leq C.
\ea\ee
\end{lemma}

\section{Proofs of Theorem \ref{th0}--\ref{th3}}
This section presents the proof of our main theorems.

\noindent\textbf{Proof of Theorem \ref{th2}}.
Let $(\n_0,u_0)$  be the initial data in Theorem \ref{th2}, 
satisfying (\ref{wsol1}), $(\ref{wsol01})_1$, $(\ref{ssol1})_1$, $(\ref{csol1})_1$, and (\ref{csol2}).
For $\de \in (0,1)$, following an approach similar to that in \cite{HL}, we may construct a sequence of smooth functions $\n_0^\de$ satisfying
\be\la{czbj1}\ba
\de \le \n_0^\de \le \| \n_0 \|_{L^\infty} + 1, \quad \frac12 \le \int_{B_{N_0}} \n^\de_0 dx
\le \int_{\rr} \n^\de_0dx \le \frac32,
\ea\ee
and
\be\la{czbj2}\ba
\begin{cases}
\bar x^a \n_{0}^{\de} \rightarrow \bar x^a \n_{0} \quad & {\rm in}\,\, L^1(\rr)\cap H^{1}(\rr)\cap W^{1,q}(\rr), \\
( \na^2 \n^\de_0,\na^2 P(\n_0^\de) ) \rightarrow ( \na^2 \n_0,\na^2 P(\n_0) ) \quad & {\rm in}\,\, L^{q}(\rr), \\
\bar{x}^{\de_0} ( \na^2 \n^\de_0,\na^2 P(\n_0^\de) ) \rightarrow ( \na^2 \n_0,\na^2 P(\n_0) ) \quad & {\rm in}\,\, L^2(\rr),
\end{cases}
\ea\ee
as $\de \rightarrow 0$.

By the local existence result in Lemma \ref{lct},
there exists a $T_\de>0$ such that the problem (\ref{ns})--(\ref{i3}) with $\tilde{\n}=0$ and the initial data $(\n_0^\de,u_0)$
has a unique classical solution $(\n^\de,u^\de)$ on $\rr \times (0,T_\de]$.
The a priori estimates Proposition \ref{l5} and Lemmas \ref{c21} and \ref{c26} ensure that this local classical solution $(\n^\de,u^\de)$ can be extended to $(0,T]$ for any $T>0$, provided $\nu \ge \nu_1$.
Furthermore, $(\n^\de,u^\de)$ satisfies all the estimates listed in Lemmas \ref{c21} and \ref{c26} with $C$ independent of $\de$.
Letting $\de \to 0$ and applying standard arguments (see \cite{HL,LX2}), we deduce that the problem (\ref{ns})--(\ref{i3}) with $\tilde{\n}=0$ has a global classical solution $(\n,u)$ satisfying the properties listed in Theorem \ref{th2} when $\nu \ge \nu_1$.
Moreover, the proof of uniqueness of $(\n,u)$ satisfying (\ref{ssol4}) and (\ref{csol3}) is similar to \cite{LLL}.

In view of Proposition \ref{1cpp1}, Lemmas \ref{lsn21}--\ref{lcn26},
and following the arguments similar to the case $\tilde{\n}=0$, we can obtain that for $\nu \ge \nu_2$, the problem (\ref{ns})--(\ref{i3}) with $\tilde{\n}>0$ has a unique global classical solution $(\n,u)$ satisfying
(\ref{2wsol2}), (\ref{ssol5}), and (\ref{csol4}).
It remains to prove (\ref{cpwsol4}).

From (\ref{1cpp2}), (\ref{1cp02}), (\ref{1cp310}), (\ref{1cp311}), and (\ref{1cp316}), we conclude that
\be\la{qkjltb1}\ba
\int_1^\infty \left( \| \n - \tilde{\n} \|^4_{L^4} + \| \na u \|^2_{L^2} + \| \na u \|^4_{L^4} \right) dt \le C.
\ea\ee
Multiplying $(\ref{ns})_1$ by $4(\n-\tilde{\n})^3$ and integrating by parts over $\rr$, we obtain
\be\la{qkjltb2}\ba
\frac{d}{dt} \left( \| \n - \tilde{\n} \|^4_{L^4} \right)
& = \int \left( (\n - \tilde{\n})^4 \div u - 4 (\n - \tilde{\n})^3 \n \div u \right) dx \\
& \le C \| \n - \tilde{\n} \|^4_{L^4} + C \| \na u \|^2_{L^2},
\ea\ee
which yields that, for all $1 \le N \le s \le N+1 \le t \le N+2$,
\be\la{qkjltb3}\ba
\| \n(\cdot,t) - \tilde{\n} \|^4_{L^4}
\le \| \n(\cdot,s) - \tilde{\n} \|^4_{L^4} + C \int_N^{N+1} \left( \| \n - \tilde{\n} \|^4_{L^4} + C \| \na u \|^2_{L^2} \right) dt.
\ea\ee
Integrating (\ref{qkjltb3}) with respect to $s$ over $[N,N+1]$ gives
\be\la{qkjltb4}\ba
\| \n(\cdot,t) - \tilde{\n} \|^4_{L^4}
\le C \int_N^{N+1} \left( \| \n - \tilde{\n} \|^4_{L^4} + C \| \na u \|^2_{L^2} \right) dt,
\ea\ee
which together with (\ref{qkjltb1}) implies
\be\la{qkjltb5}\ba
\lim_{t \to \infty} \| \n(\cdot,t) - \tilde{\n} \|_{L^4} = 0.
\ea\ee
It follows from (\ref{1cpp2}), (\ref{1cp02}), (\ref{qkjltb5}), and H\"older's inequality that for any $s \in (2,\infty)$,
\be\la{qkjltb6}\ba
\lim_{t \to \infty} \| \n(\cdot,t) - \tilde{\n} \|_{L^s} = 0.
\ea\ee
On the other hand, direct calculation shows
\be\la{qkjltb7}\ba
\int_1^\infty \left| \frac{d}{dt}(\| \na u \|^2_{L^2}) \right| dt
& = 2 \int_1^\infty \left| \int \p_i u^j \p_i u^j_t dx \right| dt \\
& = 2 \int_1^\infty \left| \int \p_i u^j \p_i \left( \dot{u}^j - u^k \p_k u^j \right) dx \right| dt \\
& = \int_1^\infty \left| \int \left( 2 \p_i u^j \p_i \dot{u}^j - 2 \p_i u^j \p_i u^k \p_k u^j + |\na u|^2 \div u \right) dx \right| dt \\
& \le C \int_1^\infty \left( \| \na \dot{u} \|^2_{L^2} + \| \na u \|^2_{L^2} + \| \na u \|^4_{L^4} \right) dt \le C,
\ea\ee
which along with (\ref{qkjltb1}) leads to
\be\la{qkjltb8}\ba
\lim_{t \to \infty} \| \na u \|_{L^2} = 0.
\ea\ee
Furthermore, using (\ref{1cp27}) and (\ref{p2}), we derive for any $2 \le r <\infty$,
\be\la{qkjltb9}\ba
\| \na u \|_{L^r}
& \le C \left( \| \div u \|_{L^r} + \| \o \|_{L^r} \right) \\
& \le C \left( \| G \|_{L^r} + \| P-P(\tilde{\n}) \|_{L^r} + \| \o \|_{L^r} \right) \\
& \le C \left( 1 + \| \na u \|_{L^2} + \| \n \dot{u} \|_{L^2} \right),
\ea\ee
which yields
\be\la{qkjltb10}\ba
\sup_{1 \le t <\infty} \| \na u \|_{L^r} \le C.
\ea\ee
From (\ref{qkjltb10}), (\ref{qkjltb8}), and H\"older's inequality, we conclude that for any $2 \le r <\infty$,
\be\la{qkjltb11}\ba
\lim_{t \to \infty} \| \na u \|_{L^r} = 0.
\ea\ee
Combining (\ref{qkjltb6}) with (\ref{qkjltb11}) yields (\ref{cpwsol4}), thus completing the proof of Theorem \ref{th2}.

\noindent\textbf{Proofs of Theorems \ref{th0} and \ref{th1}}.
By employing standard compactness arguments in \cite{F,L2,LZZ}, the proofs are similar to that of Theorem \ref{th2}, and hence are omitted.

\noindent\textbf{Proof of Theorem \ref{th01}}.
For any $0<T<\infty$, when $\nu>\nu_1$,
we deduce from (\ref{wsol2}), (\ref{cp1}), (\ref{cp032}), (\ref{pd11}) and (\ref{pt1})
that $\{\n^{\nu}\}_\nu$ is bounded in $L^\infty(0,T;L^1) \cap L^\infty(\rr \times (0,T))$, and for any $0<\tau<T$, $0<R<\infty$,
$\{u^{\nu}\}_\nu$ is bounded in $L^\infty(\tau,T;H^1(B_R))\cap L^2(0,T;H^1(B_R))$,
and $\{ \na u^{\nu}\}_\nu$ is bounded in $L^\infty(\tau,T;L^2(\rr))\cap L^2(0,T;L^2(\rr))$.

Moreover, it follows from (\ref{pt1}), (\ref{gw}), (\ref{p2}), (\ref{gn11})
and H\"older's inequality that
\be\la{ins01}\ba
\| u^{\nu}_t \|_{L^2(B_R)}
& \le C \left( \| \dot{u^{\nu}} \|_{L^2(B_R)} + \| u^{\nu} \cdot \na u^{\nu} \|_{L^2(B_R)} \right) \\
& \le C \left( \| \sqrt{\n^{\nu}} \dot{u^{\nu}} \|_{L^2(B_R)} 
+ \| \na \dot{u^{\nu}} \|_{L^2(B_R)}
+ \| u^{\nu} \|_{L^4(B_R)} \| \na u^{\nu} \|_{L^4(B_R)} \right) \\
& \le C \| \sqrt{\n^{\nu}} \dot{u^{\nu}} \|_{L^2} + C \| \na \dot{u^{\nu}} \|_{L^2}
+ \frac{C}{\nu} \| u^{\nu} \|_{H^1(B_R)} \| G^\nu \|^{1/2}_{L^2} \| \n^\nu \dot{u^{\nu}} \|^{1/2}_{L^2} \\
& \quad + C \| u^{\nu} \|_{H^1(B_R)} 
\left( \| \o^\nu \|^{1/2}_{L^2} \| \n^\nu \dot{u^{\nu}} \|^{1/2}_{L^2}
+\| P^\nu \|_{L^4} \right),
\ea\ee
which, together with (\ref{cp032}) and (\ref{pd11}), 
implies that $\{u^{\nu}\}_\nu$ is bounded in $H^1(\tau,T;L^2(B_R))$ 
for every fixed constant $R$.

Therefore, without loss of generality, we may assume
there exists $(\n,u)$ which satisfies that, for any $0<R<\infty$,
$\n \in L^\infty(\rr \times (0,T))$, $u \in L^\infty(\tau,T;H^1(B_R))\cap L^2(0,T;H^1(B_R))$,
$u_t \in L^2(B_R \times (\tau,T))$,
$ \na u \in L^\infty(\tau,T;L^2(\rr))\cap L^2(0,T;L^2(\rr)) $,
and a subsequence $(\n^n,u^n)$ of $(\n^{\nu},u^{\nu})$ such that
\be\la{ins1}\ba
\begin{cases}
\n^n \rightharpoonup \n  \mbox{ weakly * in } L^\infty(\rr \times (0,T)),\\
u^n \rightharpoonup u  \mbox{ weakly * in } \ L^\infty(\tau,T;H^1(B_R))\cap L^2(0,T;H^1(B_R)), \\
u^n \to u  \mbox{ strongly  in } \ L^\infty(\tau,T;L^p(B_R)), \\
\na	u^n \rightharpoonup \na u  \mbox{ weakly * in } \ L^\infty(\tau,T;L^2(\rr))\cap L^2(0,T;L^2(\rr)),
\end{cases}
\ea\ee 
for any $1 \le p < \infty$ and $0<R<\infty$.

Furthermore, we set $G^n:=n\div u^n-P^n$ and $\o^n:=\na^\bot \cdot u^n$.
It follows from (\ref{cp1}), Sobolev embedding and H\"older's inequality that for any $2<r<\infty$
\be\la{ins03}\ba
\| G^n \|_{L^r} 
& = \| (-\Delta)^{-1} \div (\n^n \dot{u^n}) \|_{ L^r } \\
& \le C\| \n^n \dot{u^n} \|_{ L^{\frac{2r}{r+2}} } \\
& \le C\| \sqrt{\n^n} \|_{ L^{2r} }  \| \sqrt{\n^n} \dot{u^n} \|_{L^2} \\
& \le C\| \n^n \|^{\frac{1}{2}}_{ L^{r} }  \| \sqrt{\n^n} \dot{u^n} \|_{L^2}.
\ea\ee
Combining this with (\ref{p2}) and (\ref{cp032}), 
we conclude that for any $2<r<\infty$ and $0<R<\infty$,
$\{G^n\}_n$ is bounded in $L^2(\tau,T;L^r) \cap L^2(\tau,T;H^1(B_R))$,
$\{\na G^n\}_n$ and $\{\o^n\}_n$ are bounded in $L^2(\tau,T;L^2(\rr))$.
Hence, without loss of generality, we can assume that there exist $\pi \in L^2(\tau,T;L^r) \cap L^2(\tau,T;H^1(B_R))$
and $\na \pi \in L^2(\tau,T;L^2(\rr)) $ such that
\be\la{ins3}\ba
G^n \rightharpoonup - \pi \quad  \mbox{ weakly in } L^2(\tau,T;H^1(B_R)),
\ea\ee
for any $0<R<\infty$.
Then, from $(\ref{ns})_2$, we deduce that $(\n^n,u^n)$ satisfies
\be\la{ins4}\ba
(\n^n u^n)_t+\div(\n^n u^n\otimes u^n) -\na G^n -\mu \na^{\bot} w^n =0.
\ea\ee
By taking the limit as $n \to \infty$, we obtain that $(\n,u)$ satisfies
\be\la{ins6}\ba
\begin{cases}
\n_t+\div(\n u)=0,\\
(\n u)_t+\div(\n u\otimes u) -\mu \na^{\bot} w + \na \pi =0,
\end{cases} 
\ea\ee
On the other hand, based on (\ref{cp1}), (\ref{gw}) and (\ref{cp032}), we derive
\be\la{ins7}\ba
\div u^n \to 0  \  \mbox{ strongly in } L^2(\rr \times (0,T))\cap L^\infty( (\tau,T); L^2),
\ea\ee
which together with (\ref{ins1}) yields $\div u=0$.
This fact, together with the equality $\Delta u = \na \div u + \na^{\bot} w $, leads to $\na^{\bot} w=\Delta u$.
Hence, $(\n,u)$ satisfies (\ref{isol2}) and (\ref{isol3}).

Next, if the initial data $(\n_0,u_0)$ satisfies (\ref{ws}) and (\ref{lws00}),
for any $0<T<\infty$, when $\nu \ge \nu_1$,
we deduce from (\ref{pt1}), (\ref{cp1}), (\ref{cp031}) and (\ref{wsol2})
that $\{\n^{\nu}\}_\nu$ is bounded in $L^\infty(\rr \times (0,T))$,
and $\{\na u^{\nu}\}_\nu$ is bounded in $L^\infty(0,T;L^2)$.
Moreover, by directly integrating (\ref{pd118}) over $(0,T)$ and define
$\sqrt{\n} \du (x,t=0)= \sqrt{\n_0} g_1$,
after using (\ref{pd120}) and (\ref{cp031}), we derive that
$\{\sqrt{\n^\nu} \dot{u^{\nu}}\}_\nu$ is bounded in $L^\infty(0,T;L^2)$
and $\{\na \dot{u^{\nu}}\}_\nu$ is bounded in $L^2(\rr \times (0,T))$.
In addition, from (\ref{p2}) we obtain that the sequences
$\{\na G^{\nu}\}_\nu$ and $\{\na \o^{\nu}\}_\nu$ are bounded in $L^\infty(0,T;L^2)$.
Therefore, by arguments similar to those above,
we conclude that the sequence $(\n^{\nu},u^{\nu})$ has a subsequence that converges to the global solution of (\ref{isol2}), and $(\n,u)$ satisfies (\ref{lws1}).
Then, according to the result in \cite[Proposition 4.2]{PT}, the system (\ref{isol2}) with initial data $(\n_0,u_0)$ satisfying (\ref{lws1}) has a unique global solution.
This implies that the whole sequence $(\n^{\nu},u^{\nu})$ converges to the global solution of (\ref{isol2}), and $(\n,u)$ satisfies (\ref{lws1}).
This completes the proof of Theorem \ref{th01}.

\noindent\textbf{Proof of Theorem \ref{th001}}.
For any $0<T<\infty$, when $\nu>\nu_2$,
we deduce from (\ref{wsol2}), (\ref{cp1}), (\ref{1cpp2}) and (\ref{LZ02})
that $\{\n^{\nu}\}_\nu$ is bounded in $L^\infty(\rr \times (0,T))$, and for any $0<\tau<T$,
$\{u^{\nu}\}_\nu$ is bounded in $L^\infty(\tau,T;H^1)\cap L^2(0,T;H^1)$.
Additionally, similar to (\ref{ins01}), by using
(\ref{pt1}), (\ref{gw}), (\ref{p2}), (\ref{gn11}) and Lemma \ref{LZ1} leads to
\be\la{}\ba
\| u^{\nu}_t \|_{L^2}
& \le C \| \sqrt{\n^{\nu}} \dot{u^{\nu}} \|_{L^2} + C \| \na \dot{u^{\nu}} \|_{L^2}
+ \frac{C}{\nu} \| u^{\nu} \|_{H^1} \| G^\nu \|^{1/2}_{L^2} \| \n^\nu \dot{u^{\nu}} \|^{1/2}_{L^2} \\
& \quad + C \| u^{\nu} \|_{H^1} 
\left( \| \o^\nu \|^{1/2}_{L^2} \| \n^\nu \dot{u^{\nu}} \|^{1/2}_{L^2}
+\| P^\nu -P(\tilde{\n}) \|_{L^4} \right),
\ea\ee
which along with (\ref{1cpp2}) and (\ref{1cp04}), we obtain $\{u^{\nu}\}_\nu$ is bounded in $H^1(\tau,T;L^2(B_R))$.

Therefore, with a slight abuse of notation, 
there exists a subsequence $(\n^n,u^n)$ of $(\n^{\nu},u^{\nu})$ and
$\n \in L^\infty(\rr \times (0,T)), u \in L^\infty(\tau,T;H^1)\cap L^2(0,T;H^1)$ such that
\be\la{inss1}\ba
\begin{cases}
\n^n \rightharpoonup \n  \mbox{ weakly * in } L^\infty(\rr \times (0,T)), \\
u^n \rightharpoonup u  \mbox{ weakly * in } \ L^2(0,T;H^1) \cap L^\infty(\tau,T;H^1) \\
u^n \to u  \mbox{ strongly  in } \ L^\infty(\tau,T;L^p(B_R)),
\end{cases}
\ea\ee 
for any $1 \le p < \infty$ and $0<R<\infty$.

Then, we define $G^n:=n\div u^n-(P^n-P(\tilde{\n}))$ and $\o^n:=\na^\bot \cdot u^n$.
Together with (\ref{1cpp2}) and (\ref{p2}),
this implies that $\{\na G^n\}_n$ and $\{\o^n\}_n$ are bounded in $L^2(\rr \times (\tau,T))$.
Hence, without loss of generality, we can assume that there exists $ \Psi  \in L^2(\rr \times (\tau,T))$ such that
\be\la{inss3}\ba
\na G^n \rightharpoonup \Psi \quad  \mbox{ weakly in } L^2(\rr \times (\tau,T)).
\ea\ee
In addition, for any $\phi \in \left( \mathcal{D}(\rr) \right)^2$ with $\div \phi=0$ and $\psi \in \mathcal{D}(\tau,T)$, we have
\be\la{inss6}\ba
\int _\tau ^T \int \na G^n \cdot \phi dx \ \psi \ dt=0,
\ea\ee
Then, letting $n$ tend to $\infty$, by applying (\ref{inss3}) we can get
\be\la{inss7}\ba
\int _\tau ^T \int \Psi \cdot \phi dx \ \psi \ dt=0,
\ea\ee
which implies
\be\la{inss8}\ba
\int \Psi \cdot \phi dx =0, \ \ \mathrm{a.e.\ }t \in (\tau,T).
\ea\ee

Therefore, in light of the De Rham Theorem \cite[Proposition 1.1]{TR}, we can conclude that there exists a distribution $ \pi $ such that $ \Psi = -\na \pi$.

Let $n$ tend to $\infty$, and we can obtain that $(\n,u)$ satisfies
\be\la{inss9}\ba
\begin{cases}
\n_t+\div(\n u)=0,\\
(\n u)_t+\div(\n u\otimes u) -\mu \na^{\bot} w + \na \pi =0.
\end{cases} 
\ea\ee
Similar to the arguments for $\tilde{\n}=0$, we also have $\na^{\bot} w=\Delta u$ and $\div u=0$, and
$(\n,u)$ satisfies (\ref{isol2}) and (\ref{lisol3}).

Moreover, for any $0<T<\infty$, when $\nu \ge \nu_2$ and the initial data $(\n_0,u_0)$ satisfies (\ref{0lws}).
Through a similar argument analogous to the case of $\tilde{\n}=0$,
we deduce that $(\n^{\nu},u^{\nu})$ contains a subsequence
converging to a global solution of (\ref{isol2}), and $(\n,u)$ satisfies (\ref{0lws1}).
Therefore, invoking the uniqueness result from \cite{PT} for system (\ref{isol2}) with initial data $(\n_0,u_0)$ satisfying (\ref{0lws}),
we establish that the whole sequence $(\n^{\nu},u^{\nu})$ converges to the unique global solution of (\ref{isol2}), and $(\n,u)$ satisfies (\ref{0lws1}).
This completes the proof of Theorem \ref{th001}.

\noindent\textbf{Proof of Theorem \ref{th3}}.
The proof of Theorem \ref{th3} is similar to that of \cite[Theorem 1.2]{LX}, so we omit it here.

\end{document}